\documentclass[12pt]{amsart}
\usepackage{amssymb}
\usepackage{amsmath,amscd}

\let\frak\mathfrak
\let\Bbb\mathbb

\def\>{\relax\ifmmode\mskip.666667\thinmuskip\relax\else\kern.111111em\fi}
\def\<{\relax\ifmmode\mskip-.333333\thinmuskip\relax\else\kern-.0555556em\fi}
\def\vsk#1>{\vskip#1\baselineskip}
\def\vv#1>{\vadjust{\vsk#1>}\ignorespaces}
\def\vvn#1>{\vadjust{\nobreak\vsk#1>\nobreak}\ignorespaces}

\def\sskip{\par\vskip.2\baselineskip plus .05\baselineskip}
\let\Medskip\medskip
\def\medskip{\par\Medskip}
\let\Bigskip\bigskip
\def\bigskip{\par\Bigskip}

\let\Maketitle\maketitle
\def\maketitle{\hrule height0pt\vskip-\baselineskip
\Maketitle\thispagestyle{empty}\let\maketitle\empty}

\newtheorem{thm}{Theorem}[section]
\newtheorem{cor}[thm]{Corollary}
\newtheorem{lem}[thm]{Lemma}

\numberwithin{equation}{section}

\theoremstyle{definition}

\let\mc\mathcal
\let\nc\newcommand

\nc{\on}{\operatorname}
\nc{\Z}{{\mathbb Z}}
\nc{\C}{{\mathbb C}}
\nc{\N}{{\mathbb N}}
\nc{\pone}{{\mathbb C}{\mathbb P}^1}
\nc{\arr}{\rightarrow}
\nc{\larr}{\longrightarrow}
\nc{\al}{\alpha}
\nc{\W}{{\mc W}}
\nc{\la}{\lambda}
\nc{\su}{\widehat{{\mathfrak sl}}_2}
\nc{\g}{{\mathfrak g}}
\nc{\h}{{\mathfrak h}}
\nc{\m}{{\mathfrak m}}
\nc{\n}{{\mathfrak n}}
\nc{\Gm}{\Gamma}
\nc{\La}{\Lambda}
\nc{\gl}{\widehat{\mathfrak{gl}_2}}
\nc{\bi}{\bibitem}
\nc{\om}{\omega}
\nc{\Res}{\on{Res}}
\nc{\gm}{\gamma}
\nc{\Om}{\Omega}

\def\fratop{\genfrac{}{}{0pt}1}
\def\satop#1#2{\fratop{\scriptstyle#1}{\scriptstyle#2}}

\def\ev{\mbox{\sl ev}}
\def\ch{\on{ch}}

\def\End{\on{End\>}}
\def\Gr{\on{Gr}}
\def\Res{\on{Res}}
\def\rdet{\on{rdet}}
\def\Wr{\on{Wr}}

\def\B{{\mc B}}
\def\D{{\mc D}}

\def\O{{\mc O}}

\def\V{{\mc V}}

\let\dl\delta
\let\Dl\Delta

\let\si\sigma

\let\Sig\varSigma

\let\Tilde\widetilde
\let\der\partial

\let\ge\geqslant
\let\geq\geqslant
\let\le\leqslant
\let\leq\leqslant

\nc{\gln}{\mathfrak{gl}_N}
\nc{\sln}{\mathfrak{sl}_N}

\def\glnt{\gln[t]}

\def\Uglnt{U(\glnt)}

\def\beq{\begin{equation}}
\def\eeq{\end{equation}}
\def\be{\begin{equation*}}
\def\ee{\end{equation*}}

\nc{\bean}{\begin{eqnarray}}
\nc{\eean}{\end{eqnarray}}
\nc{\bea}{\begin{eqnarray*}}
\nc{\eea}{\end{eqnarray*}}
\nc{\bs}{\boldsymbol}
\nc{\Ref}[1]{{\rm(\ref{#1})}}
\nc{\glN}{\mathfrak{gl}_N}
\nc{\glNt}{\mathfrak{gl}_N[t]}
\nc{\s}{sing}
\nc{\R}{\Bbb R}

\nc{\Oml}{{\Om_{\bs\la}}}
\nc{\OmLb}{{\Om_{\bs\La,\bs\la,\bs b}}}
\nc{\Ol}{{\mc O_{\bs\la}}}
\nc{\OLb}{{\mc O_{\bs\La,\bs\la,\bs b}}}
\nc{\VSl}{{(\V^S)_{\bs\la}}}
\nc{\Bl}{{\B_{\bs\la}}}
\nc{\Ml}{{\mc M_{\bs\la}}}
\nc{\Mlb}{{\mc M_{\bs\La,\bs\la,\bs b}}}
\nc{\Blb}{{\B_{\bs\La,\bs\la,\bs b}}}
\nc{\Omn}{{\Omega_{\bs n,\bs b,\bs K}}}
\nc{\Omlb}{{\bar\Om_{\bs\la}}}
\nc{\ep}{\epsilon}

\nc{\Dlb}{\Dl_{\bs\La,\bs\la,\bs b,\bs K}}
\nc{\Bb}{{\bf b}}
\nc{\glt}{{\frak{gl}_2}}
\nc{\A}{{\mc A}}
\nc{\slt}{{\frak{sl}_2}}
\nc{\Ma}{{\mc M_{\bs a}}}
\nc{\Mal}{{\mc M_{\bs\la,\bs a}}}
\nc{\Malp}{{\mc M_{\phi,\bs\la,\bs a}}}

\nc{\Bal}{{\B_{\bs\la,\bs a}}}
\nc{\Ola}{{\mc O_{\bs\la,\bs a}}}
\nc{\Bv}{{\mc B_{\V^S}}}
\nc{\Bvz}{{\mc B^0_{\V^S}}}

\nc{\sing}{{\rm Sing\,}}
\nc{\Uglt}{U(\glt)}
\nc{\Olo}{{\mc O^0_{\bs\la}}}
\nc{\kk}{K}

\nc{\Oll}{{\Omega_{\bs\la}}}
\nc{\T}{{\mc T}}

\nc{\CC}{{\mc C}}
\nc\Vl{{(\V^S)^{sing}_{\bs\la}}}

\def\plainlabel{\def\makelabel##1{##1}}

\begin{document}

\title[ Three sides of the geometric Langlands correspondence]
{Three sides of the geometric Langlands correspondence
for $\gln$ Gaudin model and Bethe vector averaging maps}

\author[E.\,Mukhin, V.\,Tarasov, A.\,Varchenko]
{E.\,Mukhin$\>^{*,1}$, V.\,Tarasov$\>^{\star,*}$,
and A.\,Varchenko$\>^{\diamond,2}$}

\thanks{${}^1$\ Supported in part by NSF grant DMS-0601005}
\thanks{${}^2$\ Supported in part by NSF grant DMS-0555327}

\maketitle

\begin{center}
{\it $^\star\<$Department of Mathematical Sciences
Indiana University\,--\>Purdue University Indianapolis\\
402 North Blackford St, Indianapolis, IN 46202-3216, USA\/}

\medskip
{\it $^*\<$St.\,Petersburg Branch of Steklov Mathematical Institute\\
Fontanka 27, St.\,Petersburg, 191023, Russia\/}

\medskip
{\it $^\diamond\<$Department of Mathematics, University of North Carolina
at Chapel Hill\\ Chapel Hill, NC 27599-3250, USA\/}
\end{center}

\medskip
\begin{abstract}
We consider the $\gln$ Gaudin model of a tensor power of
the standard vector representation.
The geometric Langlands correspondence in the Gaudin model relates
the Bethe algebra of the commuting Gaudin Hamiltonians and the
algebra of functions on a suitable space
of $N$-th order differential operators. In
this paper we introduce a third side of the
correspondence: the algebra
of functions on the critical set of a master function.
We construct isomorphisms of the third algebra and the first two.

A new object is the Bethe vector averaging maps.
\end{abstract}

\maketitle

\section{Introduction}

We consider the $\gln$ Gaudin model associated with a tensor power of
the standard vector representation.
The geometric Langlands correspondence identifies
the Bethe algebra of the commuting Gaudin Hamiltonians and
the algebra of functions on a suitable space
of $N$-th order differential operators. In
this paper we introduce a third ingredient of the
correspondence: the algebra
of functions on the critical set of a master function.
We construct isomorphisms of the three algebras.

Master functions were introduced in \cite{SV} to construct
hypergeometric integral solutions of the KZ equations,
\bea
\kappa \frac{\der I}{\der z_i}
= H_i(\bs z) I(\bs z) , \phantom{a} i=1,\dots,n\ ,
\qquad I(\bs z)\ =\ \int \Phi(\bs z,\bs t)^{1/\kappa}
\omega(\bs z,\bs t) d\bs t\ ,
\eea
where $H_i(\bs z)$ are the Gaudin Hamiltonians,
$\Phi(\bs z,\bs t)$ is a scalar master functions, $\omega(\bs z,\bs t)$ is a
universal weight function, which is a vector valued function.
It was realized almost immediately
\cite{Ba, RV} that the value of the universal weight function at a
critical point of the master function is an eigenvector of the Gaudin
Hamiltonians. This construction of the eigenvectors is called the
Bethe ansatz. The critical point equations for the master function
are called the Bethe ansatz equations and the eigenvectors are called
the Bethe vectors.
The Bethe ansatz gives a relation between the critical points of the master
function and the algebra generated by Gaudin Hamiltonians. The algebra of all
(in particular, generalized) Gaudin Hamiltonians is called the Bethe algebra.
Higher Gaudin Hamiltonians were introduced using different approaches in
\cite{FFR} and \cite{T}, see also \cite{MTV1}.

In \cite{ScV, MV2}, an $N$-th order differential operator was assigned to every
critical point of the master function. The differential operators appearing
in that construction form the second component of the geometric Langlands
correspondence.

The third component of the geometric Langlands correspondence is the algebra
of functions on the critical set of the master function. In this paper we show
that all three components of the geometric Langlands correspondence are
on equal footing; they are isomorphic.

The main results of the paper are Corollaries \ref{cor iso O to T} and
\ref{cor cc b}.

\medskip

The paper is organized as follows. In Section \ref{alg sec} we
recall the definition of the Bethe algebra $\B_\V$ of a tensor power of the
vector representation of $\gln$ \ \cite{MTV3}. In Section \ref{funct on schub} we
introduce the algebra $\O_W$ of functions on a suitable Schubert cell
$\W$. Points of $\W$ are some $N$-dimensional spaces of polynomials in one variable.
Such a space $X$ is characterized by a monic $N$-th order
differential operator with kernel $X$. The algebra $\O_\W$ can be
considered as the algebra of functions on the space of those
differential operators. In Section \ref{sec isom} we recall an
isomorphism
$\zeta : \O_\W \to \B_\V$ constructed in \cite{MTV3}.
In Section \ref{Critical points of the master function}
a master function
and its quotient critical set $\CC$
are introduced and an isomorphism
$\iota^* : \O_\W \to \O_\CC$ is constructed. Here $\O_\CC$ is the algebra of functions on
$\CC$.
Consequently, we obtain a composition isomorphism
$\B_\V \xrightarrow{\zeta^{-1}} \O_\W \xrightarrow{\iota^*} \O_\CC$.
In Section \ref{Universal weight function
and Bethe vectors} we introduce the universal weight function
$\omega(\bs z,\bs t)$ and describe the basic facts of the Bethe ansatz.
In Section \ref{Polynomiality} the
Bethe vector averaging maps
\bea
v_F\ :\ \bs z \ \mapsto \ \frac 1{l_1!\dots l_{N-1}!}
\sum_{(\bs z,\bs p )\in C_{\bs z}} \frac {F(\bs z,\bs p)\,\omega(\bs z,\bs p)}
{{\rm Hess}_{\bs t}\log \Phi (\bs z,\bs p)}\
\eea
are introduced.
Here $\Phi (\bs z,\bs t)$ is the master function,
$\CC_{\bs z}$ the critical set of the function $\Phi(\bs z,\,\cdot\,)$,
$\omega(\bs z,\bs t)$ the Bethe vector,
$F(\bs z,\bs t)$ an auxiliary polynomial function.
Theorem \ref{thm conj} says that
the Bethe vector averaging maps
are polynomial maps. This is the main technical result of the paper.
Using the Bethe vector averaging maps, we construct in
Section \ref{Factorized critical set and Bethe algebra}
a new (direct) isomorphism $\nu : \O_\CC \to \B_\V$.
We prove that the throughout composition
$\B_\V \xrightarrow{\zeta^{-1}} \O_\W \xrightarrow{\iota^*} \O_\CC
\xrightarrow{\nu} \B_\V$ is the identity map.
Section \ref{proofs} contains the proof of Theorem \ref{thm conj}.

\medskip
The paper discusses one example: the Gaudin model on a tensor power of
the vector representation of $\gln$.
But the
picture presented here presumably holds for more general
representations and more general Lie algebras. All
the ingredients of our considerations (the Bethe algebras, master functions,
Bethe vector averaging maps) are available in other
situations.

\medskip

The authors thank A. Gabrielov for helpful discussions.

\section{Bethe algebra $\B_{\bs\la}$}
\label{alg sec}
\subsection{Lie algebra $\gln$}
Let $e_{ij}$, $i,j=1,\dots,N$, be the standard generators of the Lie algebra
$\gln$ satisfying the relations
$[e_{ij},e_{sk}]=\dl_{js}e_{ik}-\dl_{ik}e_{sj}$.
Let $\h\subset \gln$ be the Cartan subalgebra generated by
$e_{ii}, \,i=1,\dots,N$.

Let $M$ be a $\gln$-module.
A vector $v\in M$ has weight
$\bs\la=(\la_1,\dots,\la_N)\in\C^N$ if $e_{ii}v=\la_iv$ for $i=1,\dots,N$.
A vector $v$ is singular if $e_{ij}v=0$ for $1\le i<j\le N$.
Denote by $(M)_{\bs\la}$ the subspace of $M$ of weight $\bs\la$,
by $(M)^{sing}$ the subspace of all singular vectors in $M$, and by
$(M)_{\bs\la}^{sing}$ the subspace of all singular vectors
of weight $\bs\la$.

Denote by $L_{\bs\la}$ the irreducible finite-dimensional $\gln$-module with
highest weight $\bs\la$.
The $\gln$-module $L_{(1,0,\dots,0)}$ is the standard $N$-dimensional vector
representation of $\gln$, denoted below by $V$. We choose a highest weight vector of $V$
and denote it by $v_+$.

The Shapovalov form on $V$ is the unique symmetric bilinear form
$S$ defined by the conditions
$S(v_+,v_+) = 1$, \,$S(e_{ij}u, v) = S(u, e_{ji}v)$,
for all $u, v \in V$ and $1\leq i,j\leq N$. For a natural number $n$,
the tensor Shapovalov form on $V^{\otimes n}$ is
the tensor product of the Shapovalov forms of factors.

A sequence of integers $\bs\la=(\la_1,\dots,\la_N)$ such that
$\la_1\ge\la_2\ge\dots\ge\la_N\ge0$ is called a { partition with at most
$N$ parts\/}. Denote $|\bs\la|=\la_1+\dots+\la_N$.

\subsection{Current algebra $\glnt$}
Let $\glnt=\gln\otimes\C[t]$ be the complex
Lie algebra of $\gln$-valued polynomials
with the pointwise commutator.
We identify $\gln$ with the subalgebra $\gln\otimes1$
of constant polynomials in $\glnt$. Hence, any $\glnt$-module has a canonical
structure of a $\gln$-module.

For $g\in\gln$, set
$g(u)=\sum_{s=0}^\infty (g\otimes t^s)u^{-s-1}$.
For each $a\in\C$, there exists an automorphism $\rho_a$ of $\glnt$,
\;$\rho_a:g(u)\mapsto g(u-a)$. Given a $\glnt$-module $M$, we denote by $M(a)$
the pull-back of $M$ through the automorphism $\rho_a$. As $\gln$-modules,
$M$ and $M(a)$ are isomorphic by the identity map.

We have the evaluation homomorphism,
${\ev:\glnt\to\gln}$, \;${\ev:g(u) \mapsto g\>u^{-1}}$.
Its restriction to the subalgebra $\gln\subset\glnt$ is the identity map.
For any $\gln$-module $M$, we denote by the same letter the $\glnt$-module,
obtained by pulling $M$ back through the evaluation homomorphism.

There is a $\Z_{\ge0}$-grading on $\glnt$: for any $g\in\gln$,
we have $\deg\,g\otimes t^r\,=\,r$.

\subsection{The $\glnt$-module $\V^S$}
\label{VS}

Let $n$ be a positive integer.
Let $\V$ be the space of polynomials in $z_1,\dots,z_n$ with coefficients
in $V^{\otimes n}$, $\V\>=\,V^{\otimes n}\<\otimes_{\C}\C[z_1,\dots,z_n]$.
For $v\in V^{\otimes n}$ and
$p(z_1,\dots,z_n)\in\C[z_1,\dots,z_n]$, we write
$p(z_1,\dots,z_n)\,v$ instead of $v\otimes p(z_1,\dots,z_n)$.

The symmetric group $S_n$ acts on $\V$ by permutations of the factors
of $V^{\otimes n}$ and the variables $z_1,\dots,z_n$ simultaneously,
\vvn.2>
\be
\si\bigl(p(z_1,\dots,z_n)\,v_1\otimes\dots\otimes v_n\bigr)\,=\,
p(z_{\si(1)},\dots,z_{\si(n)})\,
v_{\si^{-1}(1)}\!\otimes\dots\otimes v_{\sigma^{-1}(n)}\,,\qquad\si\in S_n\,.
\kern-3em
\vv.2>
\ee
Denote by $\V^S$ the subspace of $S_n$-invariants of $\V$.
The space $\V^S$ is a free $\C[z_1,\dots,z_n]^S$-module of rank $N^n$, see
\cite{CP}, cf. \cite{MTV3}.

The space $\V$ is a $\glnt$-module with a series $g(u)$, \,$g\in\gln$,
acting by
\vvn.1>
\beq
\label{action}
g(u)\,\bigl(p(z_1,\dots,z_n)\,v_1\otimes\dots\otimes v_n)\,=\,
p(z_1,\dots,z_n)\,\sum_{s=1}^n
\frac{v_1\otimes\dots\otimes gv_s\otimes\dots\otimes v_n}{u-z_s}\ .
\vv.2>
\eeq
The $\glnt$-action on $\V$ commutes with the $S_n$-action.
Hence, $\V^S\subset \V$ is a $\glnt$-submodule.

\medskip

Define a
grading on $\C[z_1,\dots,z_n]$ by setting $\deg z_i=1$ for all
$i$.
Define a grading on $\V$ by setting $\deg(v\otimes
p)=\deg p$ for any $v\in V^{\otimes n}$ and
$p\in\C[z_1,\dots,z_n]$. The grading on $\V$ induces a grading
on $\V^S$ and $\End(\V^S)$.
The $\glnt$-action on \/ $\V^S$ is graded,
\cite{CP}.

\medskip
Let $\bs\la$ be a partition of $n$. The space
$(\V^S)^{sing}_{\bs\la}$ is a free graded $\C[z_1,\dots,z_n]^S$-module. Its graded
character is
\bean
\label{char V}
\ch((\V^S)_{\bs\la}^{sing})\,=\,
\frac{\prod_{1\le i<j\le N}(1-q^{\la_i-\la_j+j-i})}
{\prod_{i=1}^N(q)_{\la_i+N-i}}\ q^{\sum_{i=1}^N{(i-1)\la_i }},
\eean
where\/ $\,(q)_a=\prod_{j=1}^a(1-q^j)\,$, see \cite{CP}, \cite{CL}, \cite{MTV3}.

\subsection{Bethe algebra}
\label{secbethe}
Given an ${N\times N}$ matrix $A=(a_{ij})$,
we define its row determinant to be
\vvn.3>
\bea
\rdet A\,=
\sum_{\;\si\in S_N\!} (-1)^\si\,a_{1\si(1)}a_{2\si(2)}\dots a_{N\si(N)}\,.
\vv.2>
\eea
Let $\der$ be
the operator of differentiation in the variable $u$.
Define the {\it universal differential operator\/} $\D$ by the formula
\vvn-.6>
\be
\D=\,\rdet\left( \begin{matrix}
\der-e_{11}(u) & -\>e_{21}(u)& \dots & -\>e_{N1}(u)\\
-\>e_{12}(u) &\der-e_{22}(u)& \dots & -\>e_{N2}(u)\\
\dots & \dots &\dots &\dots \\
-\>e_{1N}(u) & -\>e_{2N}(u)& \dots & \der-e_{NN}(u)
\end{matrix}\right).
\vv.2>
\ee
It is a differential operator in $u$, whose coefficients are
formal power series in $u^{-1}$ with coefficients in $\Uglnt$,
\vvn-.3>
\bea
\D=\,\der^N+\sum_{i=1}^N\,B_i(u)\,\der^{N-i}\>,
\qquad
B_i(u)\,=\,\sum_{j=i}^\infty B_{ij}\>u^{-j}\,,
\eea
and $B_{ij}\in\Uglnt$, \,$i=1,\dots,N$, \,$j\in\Z_{\ge i}\>$.
The unital subalgebra of $\Uglnt$ generated by $B_{ij}$,
\,$i=1,\dots,N$, \,$j\in\Z_{\ge 0}\>$, is called the {\it Bethe algebra\/}
and denoted by $\B$.

\sskip
By \cite{T}, cf.~\cite{MTV1},
the algebra $\B$ is commutative,
and $\B$ commutes with the subalgebra $U(\gln)\subset \Uglnt$.

\subsubsection{}
\label{ass with eigen}
Let $M$ be a $\B$-module and $v\in M$ an eigenvector of $\B$.
For every coefficient $B_i(u)$ we have $B_i(u)v=h_i(u)v$,
where $h_i(u)$ is a scalar series. The scalar differential operator
$
\D_v\>=\,\der^N+\>\sum_{i=1}^N\>h_i(u)\,\der^{N-i}
$
will be called
the differential operator associated with an eigenvector $v$.

\subsubsection{}
As a subalgebra of $\Uglnt$, the algebra $\B$ acts on any $\glnt$-module $M$.
Since $\B$ commutes with $U(\gln)$, it preserves the weight subspaces of $M$
and the subspaces $(M)^{sing}_{\bs\la}$.

For a $\B$-module $M$, the image of $\B$ in $\End(M)$ is called
the Bethe algebra of $M$.

\subsubsection{}
\label{sec B^V_i}
Let $\bs\la$ be a partition of $n$ with at most $N$ parts. The space
$(\V^S)^{sing}_{\bs\la}$ is a $\B$-module.
Set
\bea
\D^\V=\,\der^N+\sum_{i=1}^N\,B_{i}^\V(u)\,\der^{N-i}\>,
\qquad
B_{i}^\V(u)\,=\,\sum_{j=i}^\infty B_{ij}^\V\>u^{-j}\,,
\eea
where $B^\V_{ij}$ is the image of $B_{ij}$ in $\End((\V^S)^{sing}_{\bs\la})$.

For any $(i,j)$, the element $B^\V_{ij}$ is homogeneous
of degree $j-i$.
For any $i$
the series $B^\V_{i}(u)$ is homogeneous of degree $-i$, see
\cite{MTV3}.

Denote by $\B_{\V}$ the Bethe algebra of $(\V^S)^{sing}_{\bs\la}$.
The Bethe algebra $\B_\V$ is our first main object.

\section{Algebra of functions $\mc O_\W$}
\label{funct on schub}

\subsection{Cell $\W$ and algebra $\O_{\W}$}
\label{Ominfty}
Let $N,d\in\Z_{>0}$, $N\leq d$. Let $\C_d[u]$ be the space of
polynomials in $u$ of degree less than $d$, $\dim \C_d[u]=d$.
Let $\Gr(N,d)$ be the Grassmannian of all $N$-dimensional
vector subspaces of
$\C_d[u]$.

\medskip
Given a partition $\bs\la=(\la_1,\dots,\la_N)$ with $\la_1\leq d-N$,
introduce a sequence
\vvn.2>
\be
P\,=\,\{d_1>d_2>\dots>d_N\}\,,\qquad d_i=\la_i+N-i\,.
\ee
Denote by $\W$ the subset of $ \Gr(N,d)$ consisting of all
$N$-dimensional subspaces
$X\subset\C_d[u]$ such that
for every $i=1,\dots,N$, the subspace $X$ contains a polynomial
of degree $d_i$.

In other words, $\W$ consists
of subspaces $X\subset\C_d[u]$ with a basis
$f_1(u),\dots,f_N(u)$ of the form
\vvn-.5>
\bean
\label{Basis}
f_i(u)=u^{d_i}+\sum_{j=1,\ d_i-j\not\in P}^{d_i}f_{ij}u^{d_i-j}.
\vv.2>
\eean
For a given $X\in\W$, such a basis is unique. The basis
$f_1(u),\dots,f_N(u)$ will be called the flag basis of $X$.

The set $\W$ is a (Schubert) cell isomorphic to an affine space
of dimension $|\bs\la|$ with coordinate functions $f_{ij}$.
Let $\O_{\W}$ be the algebra of regular functions on
$\W$\,,
\bea
\label{Ola}
\O_{\W}\ =\
\C[\>f_{ij}\>,\ i=1,\dots,N,\ j=1,\dots,d_i,\ d_i-j\not\in P\>]\,.
\eea
We may regard the polynomials $f_i(u)$, $i=1,\dots,N$, as generating functions
for the generators $f_{ij}$ of the algebra $\O_{\W}$.

The algebra $\O_{\W}$ is graded with $\deg\,f_{ij}\,=\,j$.
A polynomial $f_i(u)$ is homogeneous of degree $d_i$.
The graded character of $\O_{\W}$ is \
\bean
\label{char O}
\ch(\mc
O_{\W})\,=\,\frac{\prod_{1\leq i<j\leq N}\,(1-q^{d_i-d_j})}
{\prod_{i=1}^N(q)_{d_i}}\ =\
\frac{\prod_{1\le i<j\le N}(1-q^{\la_i-\la_j+j-i})}
{\prod_{i=1}^N(q)_{\la_i+N-i}}\ ,
\eean
see \,\cite{MTV3}.

\subsection{New generators of $\O_{\W}$}
For $g_1,\dots,g_N \in \C[u]$, introduce the Wronskian
\be
\Wr(g_1(u),\dots,g_N(u))\,=\,
\det\left(\begin{matrix} g_1(u) & g_1'(u) &\dots & g_1^{(N-1)}(u) \\
g_2(u) & g_2'(u) &\dots & g_2^{(N-1)}(u) \\ \dots & \dots &\dots & \dots \\
g_N(u) & g_N'(u) &\dots & g_N^{(N-1)}(u)
\end{matrix}\right),
\ee
where an $i$-th row is formed by derivatives of $g_i$.

Let $f_i(u)$, $i=1,\dots,N$, be the generating functions in \Ref{Basis}.
We have
\beq
\label{Wr coef}
\Wr(f_1(u),\dots,f_N(u))\,=\prod_{1\le i<j\le N}(d_j-d_i)
\ \Bigl(u^n+\sum_{s=1}^n (-1)^s\>A_s\,u^{n-s}\Bigr)\,,
\eeq
where $n= |\bs\la|$ and $A_1,\dots,A_n$ are elements of $\O_{\W}$.
Define
\vvn.2>
\bea
\label{DOla}
\D^\W=\,\frac{1}{\Wr(f_1(u),\dots,f_N(u))}\,\rdet
\left(\begin{matrix} f_1(u) & f_1'(u) &\dots & f_1^{(N)}(u) \\
f_2(u) & f_2'(u) &\dots & f_2^{(N)}(u) \\ \dots & \dots &\dots & \dots \\
1 & \der &\dots & \der^N
\end{matrix}\right).
\eea
We have
\vvn-.4>
\beq
\label{DO}
\D^\W=\,\der^N+\sum_{i=1}^N\,B^\W_{i}(u)\,\der^{N-i}\>,
\qquad
B^\W_i(u)\,=\,\sum_{j=i}^\infty B^\W_{ij}\>u^{-j}\,,
\eeq
and $B^\W_{ij}\in\O_{\W}$\,, \ $i=1,\dots,N$, \,$j\in\Z_{\geq i}\>$.
For any $(i,j)$, the element $B^\W_{ij}$ is homogeneous
of degree $j-i$.
For any $i$
the series $B^\W_{i}(u)$ is homogeneous of degree $-i$.
The elements $B^\W_{ij}\in\O_{\W}$, $i=1,\dots,N$,
$j\in \Z_{\geq i}$, generate the algebra $\O_{\W}$,\ see
\cite{MTV3}.

\subsubsection{}
For $X\in\W$, denote by $\D_X$ the monic scalar differential operator of
order $N$ with kernel $X$.
We call $\D_X$ the differential operator associated with $X$.
The operator $\D_X$ is obtained from $\D^\W$ by specialization of variables
$f_{ij}$ to their values at $X$.

\subsection{Wronski map}
\label{wronski}
Let $X\in\W$.
The Wronskian determinant of a basis of the subspace $X$ does not depend on the choice
of the basis up to multiplication by a number. The monic
polynomial representing the Wronskian determinant of a basis of $X$
is called the Wronskian of\/ $X$ and
denoted by $\Wr_X(u)$.

The Wronski map\
$ \W \to\C^n $ sends a point $X\in \W$ to a point $\bs a=(a_1,\dots,a_n)$,
if $\Wr_X(u)\>=\>u^n+\sum_{s=1}^n (-1)^s a_s u^{n-s}$.
The Wronski map has finite degree.

\section{Isomorphism of $\B_{\V}$ and $\O_{\W}$}
\label{sec isom}

\begin{thm} [\cite{MTV3}]
\label{first}
The map
\vvn-.1>
\be
\zeta :\O_{\W}\to\B_{\V}\,,
\qquad
B_{ij}^\W\mapsto B_{ij}^\V\,,
\vv.2>
\ee
is a well-defined isomorphism of graded algebras.
\end{thm}

The degrees of elements of $(\V^S)^{sing}_{\bs\la}$ are not less than
$\sum_{i=1}^N(i-1)\la_i$ and the homogeneous component of
$(\V^S)^{sing}_{\bs\la}$ of degree $\sum_{i=1}^N(i-1)\la_i$ is
one-dimensional, see formula \Ref{char V}. Let $v_1\in
(\V^S)^{sing}_{\bs\la}$ be a nonzero vector of degree
$\sum_{i=1}^N(i-1)\la_i$.

\begin{thm} [\cite{MTV3}]
\label{first1}
The map
\bea
\eta \ :\ \O_{\W}\ \to\ (\V^S)^{sing}_{\bs\la}\ ,
\qquad
B^\W_{ij} \ \mapsto\ B^\V_{ij}v_1\ ,
\eea
is an isomorphism
of degree $\sum_{i=1}^N(i-1)\la_i$ of graded vector spaces. The maps\/
$\zeta$ and\/ $\eta$ intertwine the action of the multiplication
operators on $\O_{\W}$ and the action of the Bethe algebra $\B_{\V}$
on $(\V^S)^{sing}_{\bs\la}$, that is, for any $F,G\in\O_{\W}$, we have
\beq
\eta(FG)\,=\,\zeta(F)\,\eta(G)\,.
\eeq
\end{thm}

\section{Critical points of the master function}
\label{Critical points of the master function}

\subsection{Master function}
\label{Master function}
Let $\bs\la=(\la_1,\dots,\la_N)$ be a partition of $n$.
Set\ $l_a\,=\sum_{b=a+1}^N \la_b\,,$ \ $a=0,\dots,N\,,$\
where $l_0=n$ \,and \,$l_N=0$. Denote $l=l_0+\dots +l_{N-1}$,
\ $\bs l=(l_0,\dots,l_{N-1})$.
Consider a set of $l$ variables
\bea
&
\bs T = (t^{(0)}_{1},\dots,t_{l_0}^{(0)},t^{(1)}_{1},\dots,t_{l_1}^{(1)},\dots,
t^{(N-1)}_{1},\dots,t_{l_{N-1}}^{(N-1)})\
\eea
and its subsets
$\bs t^0 = (t^{(0)}_{1},\dots,t_{l_0}^{(0)})$\ and\
$\bs t = (t^{(1)}_{1},\dots,t_{l_1}^{(1)},\dots,
t^{(N-1)}_{1},\dots,t_{l_{N-1}}^{(N-1)})$.
Consider the affine space $\C^l_{\bs T}=\C^l$ with coordinates
$\bs T = (\bs t^0,\bs t)$. The rational function $\Phi : \C^l\to\C$,
\bean
\label{master}
\Phi (\bs T) =
\prod_{a=1}^{N-1}\prod_{1\leq i<j\leq l_a} (t_i^{(a)}-t_j^{(a)})^{2}
\prod_{a=0}^{N-2}\prod_{i=1}^{l_a}\prod_{j=1}^{l_{a+1}}
(t_i^{(a)}-t_j^{(a+1)})^{-1} \
\eean
is called a {\it master function.}
The master functions arise in the hypergeometric solutions
of the KZ equations, see
\cite{M, SV, V1} and in the Bethe ansatz method for the Gaudin
model, see \cite{Ba, RV}.

The product of symmetric groups $S_{\bs l}=S_{l_0}\times \dots \times
S_{l_{N-1}}$ acts on the coordinates $\bs T$ by permutations of the coordinates with the
same upper index. The master function is
\linebreak
$S_{\bs l}$-invariant.

We consider the master function as a function of $\bs t$ depending on
the parameters $\bs t^{(0)}$.

A point $\bs T = (\bs t^{0},\bs t)\in \C^l$ is called a {\it critical
point} of $\log \Phi(\,\bs t^0\,,\cdot)$ if
\bea
\frac{\partial \phantom{a} }{\partial t_i^{(a)}} \log \Phi (\bs T)
\ =\ 0\ ,
\qquad
a = 1 , \dots , N-1,\quad i = 1 , \dots , l_a\ .
\eea
That is, a point $\bs T$ is a critical point
if the following system of $l-n$ equations is satisfied:
\bean
\label{BAE Gaud}
\sum_{j=1}^{l_{a-1}}\frac 1{t^{(a)}_i - t^{(a-1)}_{j}}\;-\,
\sum_{\satop{j=1}{j\neq i}}^{l_a}\frac 2{t^{(a)}_i - t^{(a)}_{j}}\;+\,
\sum_{j=1}^{l_{a+1}}\frac 1{t^{(a)}_i - t^{(a+1)}_{j}}\;=\,0\,,
\vv-.2>
\eean
here $a=1,\dots,N-1$, \,$j=1,\dots,l_a$.
In this definition we assume that all the denominators in
\Ref{BAE Gaud} are nonzero. In the Gaudin model, equations \Ref{BAE Gaud}
are called the Bethe ansatz equations.
For a point $\bs T\in\C^l$, denote
\bea
{\rm Hess}_{\bs t}\log \Phi (\bs T)\ = \ \det \left(\frac{\der^2\phantom{aaa}}
{\der t^{(a)}_i\der t^{(b)}_j}
\log \Phi (\bs T)\right)\ ,
\eea
where we take the determinant of the ${(l-n)\times(l-n)}$ matrix of second
derivatives of the function $\log \Phi$ with respect to all of the variables
$t^{(a)}_i$ with $a>0$.

For a fixed $\bs t^0$, the function $\log
\Phi(\,\bs t^0\,,\cdot)$ has finitely many critical
points, see \cite{ScV, MV2, MV3}.

\begin{thm}[\cite{ScV, MV3}]
\label{thm BA completness}
For generic $\bs t^0\in\C^n$, all critical points of the function
\linebreak
$\log \Phi(\,\bs t^0\,,\cdot)$ are nondegenerate. The number
of the $S_{l_1}\times \dots\times S_{l_{N-1}}$-orbits
of critical points equals $\dim\,(V^{\otimes n})^{sing}_{\bs\la}$.
\end{thm}

Denote by $\Tilde C\subset \C^l_{\bs T}$ the union of all
critical points of the functions $\log \Phi(\,\bs t^0\,,\cdot)$
for all $\bs t^0\in \C^n$ with distinct coordinates $t^{(0)}_1,\dots,t^{(0)}_n$.
Denote by $C \subset \C^l_{\bs T}$ the Zariski closure of $\Tilde C$. The set $C$ is
$S_{\bs l}$-invariant.

\subsection{Factorization by $S_{\bs l}$} For $a=0,\dots,N-1$, let
$\sigma_1^{(a)},\dots,\sigma_{l_a}^{(a)}$ be the elementary symmetric functions
of $t_1^{(a)},\dots,t_{l_a}^{(a)}$. Denote by $\C^l_{\bs \Sig}=\C^l$ the affine
space with coordinates
\bea
\bs\Sigma = (\sigma_1^{(0)},\dots,\sigma_{l_0}^{(0)}, \sigma_1^{(1)},\dots,\sigma_{l_1}^{(1)}, \dots,
\sigma_1^{(N-1)},\dots,\sigma_{l_{N-1}}^{(N-1)})\ .
\eea
The space $\C^l_{\bs \Sig}$ is the quotient of
$\C^l_{\bs T}$ by the $S_{\bs l}$-action.

Denote by $\CC$ the image of $C$ under the natural projection $\C^l_{\bs T}\to
\C^l_{\bs \Sig}$. The set $\CC$ will be called the {\it
quotient critical set} of the master function. Let $\mc O_\CC$ be the algebra
of regular functions on $\CC$, that is, the restriction of $\C[\bs\Sigma]$ to
$\CC$.

The algebra $\C[\bs T]$ is a graded algebra with $\deg\,t^{(a)}_i=1$
for all $(a,i)$. The algebra $\C[\bs \Sigma]$ is a graded algebra
with $\deg\,\sigma^{(a)}_i=i$ for all $(a,i)$. Equations \Ref{BAE
Gaud} are homogeneous. Hence, $\CC$ is a quasi-homogeneous algebraic
set and the algebra $\mc O_\CC$ has a grading with
$\deg\,(\sigma_i^{(a)}|_{\CC})=i$.

\subsection{A map $\theta : \W \to \C^l_{\bs \Sig}$}
\label{sec map f omega cS}
For $X\in\W$\>, let $f_{1,X}(u),\dots,f_{N,X}(u)$ be the flag basis of $X$.
Introduce the polynomials
$y_{0,X}(u)\>,\,$ $y_{1,X}(u)\>,\,\dots\,,\,$ $y_{N-1,X}(u)$ \,by the formula
\vvn.3>
\be
y_{a,X}(u)\!\!\prod_{a<i<j\leq N} (d_i-d_j)\,
\,=\,\Wr(f_{a+1,X}(u),\dots,f_{N,X}(u))\,,\qquad a=0,\dots,N-1\,.
\ee
For each $a$, the polynomial $y_{a,X}(u)$ is a monic polynomial of
degree $l_a$, $y_{a,X}(u) = u^{l_a} + \sum_{i=1}^{l_a}(-1)^i\sigma^{(a)}_{i,X}u^{l_a-i}$.
Denote by $t_{1,X}^{(a)},\dots,t_{l_a,X}^{(a)}$ the roots of $y_{a,X}(u)$.
Then $\sigma_{1,X}^{(a)},\dots,\sigma_{l_a,X}^{(a)}$ are the elementary
symmetric functions of $t_{1,X}^{(a)},\dots,t_{l_a,X}^{(a)}$. The sequence
\beq
\label{t z}
\bs T_X\,=\,(t_{1,X}^{(0)},\dots,t_{l_0,X}^{(0)},\,\dots\,,
t_{1,X}^{(N-1)},\dots,t_{l_{N-1},X}^{(N-1)})\,
\eeq
will be called the {\it root coordinates} of $X$.
For every $a$ the numbers $t_{1,X}^{(a)},\dots,t_{l_a,X}^{(a)}$ are determined up to
a permutation.
Let $\bs \Sigma_X$ be the image
of $\bs T_X$ in $\C^l_{\bs \Sig}$.

A point $X\in\W$ will be called {\it nice} if all roots of the
polynomials $y_{0,X}(u)\>,\,y_{1,X}(u)\>,\,\,\dots\,,$
$\,y_{N-1,X}(u)$ are simple and for each $a=1,\dots,N-1$, the
polynomials $y_{a-1,X}(u)$ and $y_{a,X}(u)$ do not have common roots.
Nice points form a Zariski open subset of\/ $\W$,\ see \cite{MTV6}. If $X$
is nice, then the root coordinates $\bs T_X$ satisfy the critical
point equations \Ref{BAE Gaud}, see~\cite{MV2}.

Define a polynomial map
$\theta : \W \to \C^l_{\bs \Sig}$,
$X \mapsto \bs\Sigma_X$\;.
This map induces a graded algebra homomorphism
$\C[\bs\Sigma] \to \O_\W$.

\begin{lem}
\label{lem inclusion}
We have $\theta(\W) \subset \CC$.
\end{lem}
\begin{proof}

The lemma follows from the fact that the
nice points of $\W$ are mapped to $\CC$.
\end{proof}

\subsection{Differential operator $\D^\T$ and a map
$\iota : \CC \to \W$}
\label{Differential operator and a map}
Set
\be
\chi^a(u, \bs T)\,=\,\sum_{j=1}^{l_{a-1}}\,\frac1{u-t^{(a-1)}_j}\;-\,
\sum_{i=1}^{l_a}\,\frac 1 {u- t^{(a)}_j}\;,\qquad a=1,\dots,N\,,
\ee
and
\vvn-.5>
\be
\D^\T\,=\,\bigl(\der -\chi^1(u,\bs T)\bigr)\,\dots\,
\bigl(\der-\chi^N(u,\bs T)\bigr)\,.
\ee
We have
\vvn-.3>
\bea
\D^\T=\,\der^N+\sum_{i=1}^N\,B_i^\T(u)\,\der^{N-i}\>,
\qquad
B_i^\T(u)\,=\,\sum_{j=i}^\infty B_{ij}^\T\>u^{-j}\,,
\eea
and $B_{ij}^\T\in\C[\bs T]^{S_{\bs l}}=\C[\bs\Sigma]$, \,$i=1,\dots,N$, \,$j\in\Z_{\ge i}\>$.
For a point $\bs T \in \C^l$, denote by $\D_{\bs T}$ the specialization of
$\D^\T$ at $\bs T$. We call $\D_{\bs T}$ the differential operator associated
with a point $\bs T$.

If $\bs T=(\bs t^0,\bs t)\in \C^l$ is a critical point of
$\log \Phi(\,\bs t^0\,,\cdot)$, then the kernel $X_{\bs T}$
of $\D_{\bs T}$ consists of polynomials;
moreover, $X_{\bs T}$ is a point of $\W$, see \cite{MV2}.
The correspondence $\bs T \mapsto X_{\bs T}$ defines a rational map
$\iota : \CC \to \W$.

\subsection{Quotient critical set is a nonsingular subvariety}

\begin{thm}
\label{thm nonsing}
The quotient critical set
$\CC\subset \C^l_{\bs \Sig}$
is a nonsingular subvariety.
The map $\theta:\W\to\C^l_{\bs \Sig}$ is an embedding with
$\theta (\W) = \CC$.
The map $\iota : \CC \to \W$ is an isomorphism
and $\iota\theta\,=\,{\rm id}_{\W}$.
\end{thm}

\begin{proof}
The map $\theta$, considered as a map from $\W$ to $\theta(\W)$ is finite.
The set $\theta(\W)$ is Zariski closed since $\W$ is Zariski closed.
We know from \cite{MV2} that $\theta(\W)$ contains the subset $\Tilde\CC\subset\CC$,
the image
of nondegenerate critical points.
We have $\theta(\W)=\CC$, since $\CC$ is the
Zariski closure of $\Tilde\CC$ and $\theta(\W)$ is Zariski
closed.

The fact that $\iota\theta\,=\,{\rm id}_{\W}$ at generic points of $\W$
is proved in \cite{MV2}. Therefore,
$\iota\theta\,=\,{\rm id}_{\W}$ for all points of $\W$.

Consider the algebra homomorphism $\iota^* : \O_\W \to \O_\CC$ induced by
$\iota$.
Under the map $\iota^*$ the elements $B^\W_{ij}$ are mapped to the polynomials
$B^\T_{ij} \in \C[\bs\Sigma]$ restricted to $\CC$.
Since the elements $B^\W_{ij}$ generate $\O_\W$, the map
$\theta : \W\to \C^l_{\bs \Sig}$ is an embedding and
the map $\iota : \CC\to \W$ is an isomorphism.
\end{proof}

\begin{cor}
\label{cor iso O to T}
The map $\iota^* : \O_\W \to \O_\CC,\ B^\W_{ij} \mapsto B^\T_{ij}|_\CC$,
is an isomorphism of graded algebras. In particular, the elements
$B^\T_{ij}|_\CC$ generate $\O_\CC$.
\end{cor}

\section{Universal weight function and Bethe vectors}
\label{Universal weight function and Bethe vectors}

We remind a construction of a rational map $\omega:\C^{l}\to
(V^{\otimes n})_{\bs\la}$, called the {\it universal weight function},
see~\cite{M, SV}, cf. \cite{RSV}.

A basis of $V^{\otimes n}$ is formed by the vectors
$e_Jv\,=\,e_{j_1,1}v_+\otimes \dots\otimes e_{j_n,1}v_+\,,
$\
where $J=(j_1,\dots,j_n)$ and $1\leq j_a\leq N$ for $a=1,\dots,N$. A basis
of $(V^{\otimes n})_{\bs\la}$ is formed by the vectors $e_Jv$ such that
$\#\{a\ |\ j_a>i\}\,=\,l_i$ for every $i=1,\dots,N-1$.
Such a multi-index $J$ will be called admissible.

The universal weight function has the form
$\omega(\bs T)\,=\,\sum_J\,\omega_J(\bs T) e_Jv$\
where the sum is over the set of all admissible $J$,
and the functions $\omega_J(\bs T)$ are defined below.

\sskip
For an admissible $J$ and \>$i=1,\ldots,N-1$, define
$
A_i(J)\,=\,\{\,a\ |\ 1\le a\le n\,,\ \ 1\le i<j_a\,\}\,.$\
Then $|\>A_i(J)\>|\,=\,l_i$.

\sskip
Let $\Gamma(J)$ be the set of sequences \,$\bs\gamma=
(\gamma_1,\dots,\gamma_{N-1})$
of bijections \,$\gamma_i:A_i(J)\to\{1,\dots,l_i\}$, \>$i=1,\dots,N-1$.
Then $|\>\Gamma(J)\>|\>=\prod_{i=1}^{N-1}l_i!$~.

\sskip
For $a\in A_1(J)$ and $\bs\gamma\in \Gamma(J)$, introduce a rational function
\bea
\omega_{a,\bs\gamma}(\bs T)\,=\,\frac1{t^{(1)}_{\gamma_1(a)}-t^{(0)}_a}\;
\prod_{i=2}^{j_a-1}\frac1{t^{(i)}_{\gamma_i(a)}-t^{(i-1)}_{\gamma_{i-1}(a)}}\ .
\eea
Define
\bean
\label{wf}
\omega_J(\bs T)\,=\,
\sum_{\bs\gamma\in \Gamma(J)}\,\prod_{a\in A_1(J)}\,\omega_{a,\bs\gamma} \ .
\eean

\begin{thm}
\label{thm X to Vn}
Let $\bs T=(\bs z,\bs p)$ be a nondegenerate critical point
of the function $\log \Phi(\,\bs z\,,\cdot)$, here $\bs z=(z_1,\dots,z_n)$ lies in
$\C^n$ and $\bs p$ lies in $\C^{l-n}$ with coordinates $t^{(i)}_j$, $i>0$.
Consider the value $\omega(\bs z,\bs p)$ of the universal weight function
$\omega:\C^{l}\to(V^{\otimes n})_{\bs\la}$ at\/ $(\bs z,\bs p)$. Consider
$V^{\otimes n}$ as the $\glNt$-module $\otimes_{s=1}^nV(z_s)$\>.
Then
\begin{enumerate}
\item[(i)]
The vector $\omega(\bs z,\bs p)$ belongs to $(V^{\otimes n})_{\bs\la}^{sing}$.

\item[(ii)] The vector $\omega(\bs z,\bs p)$ is an eigenvector of the Bethe algebra
$\B$, acting on $\otimes_{s=1}^nV(z_s)$. Moreover,
$\D_{\omega(\bs z,\bs p)}=\D_{(\bs z,\bs p)}$, where $\D_{\omega(\bs z,\bs p)}$
and $\D_{(\bs z,\bs p)}$ are the differential operators associated with the eigenvector
$\omega(\bs z,\bs p)$ and the point $(\bs z,\bs p)\in \C^l$, respectively, see Sections
\ref{ass with eigen} and \ref{Differential operator and a map}.
\item[(iii)]
Let $S$ be the tensor Shapovalov form on $V^{\otimes n}$,
then
\bea
S(\omega(\bs z,\bs p),\omega(\bs z,\bs p))\ = \
{\rm Hess}_{\bs t}\log \Phi (\bs z,\bs p)\ .
\eea
\item[(iv)]
If $(\bs z,\bs p)$ and $(\bs z,\bs p')$ are two
nondegenerate critical points of the function $\log \Phi(\,\bs
z\,,\cdot)$, which lie in different $S_{l_1}\times \dots\times
S_{l_{N-1}}$-orbits, then $S(\omega(\bs z,\bs p),\omega(\bs z,\bs p'))\,=\,0$.

\end{enumerate}
\end{thm}

Part (i) is proved in~\cite{Ba} and~\cite{RV}. Part (i) also follows
directly from Theorem~6.16.2 in~\cite{SV}. Part (ii) is proved in
\cite{MTV1}. Part (iii) is proved in \cite{MV3, V2}. Part (iv) is
proved in \cite{V2} and also follows from \cite{MTV3}.

The vector $\omega(\bs z,\bs p)$ is called the {\it
Bethe vector} corresponding
to a critical point $(\bs z,\bs p)$.

\section{Bethe vector averaging maps}
\label{Polynomiality}

\label{sec map v_f}
Consider $\C^l$ with coordinates $\bs T = (\bs t^0,\bs t)$.
We denote the variables $\bs t^0=(t_1^{(0)},\dots,t_n^{(0)})$
also by $\bs z = (z_1,\dots,z_n)$.

Let $\bs z$ be a generic point of $\C^n$ with distinct coordinates and
such that all the critical points of the function
$\log \Phi(\,\bs z\,,\cdot)$ are nondegenerate.
The critical set $C_{\bs z}$
of
$\log \Phi(\,\bs z\,,\cdot)$
consists of $\dim \,(V^{\otimes n})^{sing}_{\bs\la}$
$S_{l_1}\times \dots\times S_{l_{N-1}}$-orbits.
Each orbit has $l_1!\cdots l_{N-1}!$ points.
For any $F \in \C[\bs T]^{S_{\bs l}}=\C[\bs\Sigma]$, \ let us define
\bean
\label{B sum}
v_F(\bs z)\ = \ \frac 1{l_1!\dots l_{N-1}!}
\sum_{(\bs z,\bs p )\in C_{\bs z}} \frac {F(\bs z,\bs p)\,\omega(\bs z,\bs p)}
{{\rm Hess}_{\bs t}\log \Phi (\bs z,\bs p)}\ .
\eean
The term of this sum corresponding to a critical point $(\bs z,\bs p)$
can be written as the following integral, see Chapter 5 of
\cite{GH}. Choose a small neighborhood $U$ of $\bs p$ in
$\C^{l-n}$. Define a torus $\Gamma_{\bs z,\bs p}$ in $U$ by $l-n$
equations $|\Phi^a_j(\bs z,\bs t)|=\epsilon^a_j$ where $\Phi^a_j$ are
derivatives of $\log \Phi(\,\bs z\,,\cdot)$ with respect to the
variables $t^{(a)}_j$, $a>0$, and where $\epsilon^a_i$ are small
positive numbers. Then
\bean
\label{integral}
\frac {F(\bs z,\bs p)\,\omega(\bs z,\bs p)}
{{\rm Hess}_{\bs t}\log \Phi (\bs z,\bs p)}\ =\
\frac 1{(2\pi i)^{l-n}} \int_{\Gamma_{\bs z,\bs t}}
\frac {F(\bs z,\bs t)\,\omega(\bs z,\bs t)\,d\bs t}
{\prod_{a,j} \Phi_j^a(\bs z,\bs t)}\ .
\eean
The $l_1!\cdots l_{N-1}!$ terms of the sum in \Ref{B sum}
corresponding to a single $S_{l_1}\times \dots\times
S_{l_{N-1}}$-orbit are all equal due to the $S_{l_1}\times
\dots\times S_{l_{N-1}}$-invariance of $\Phi$, $\omega$ and $F$.

The correspondence $\bs z \mapsto v_F(\bs z)$ defines a map
$ v_F: \C^n \to (V^{\otimes n})^{sing}_{\bs\la}$ which will be called
{\it a Bethe vector averaging map}.

The map $v_F$ is a rational map. Indeed, the map
is well defined on a Zarisky open subset of
$\C^n$ and has bounded growth as the argument approached the possible singular points
or infinity.

\begin{thm}
\label{thm conj}
For any $F \in \C[\bs\Sigma]$, the Bethe vector averaging map $v_F$
is a polynomial map.
\end{thm}

Theorem \ref{thm conj} is proved in Section \ref{proofs}.

\section{Quotient critical set and Bethe algebra}
\label{Factorized critical set and Bethe algebra}

\subsection{Construction of isomorphisms}
Recall that $\C[\bs\Sigma]$ is graded by $\deg\sigma^{(a)}_j=j$.
For any $F \in \C[\bs\Sigma]$,
consider the Bethe vector averaging map $v_F :\C^n \to (V^{\otimes n})^{sing}_{\bs\la}$.

\begin{lem}
\label{lem homog}
If $F$ is quasi-homogeneous and $\deg F =d$, then $v_F$ is homogeneous and
\\ $\deg v_F\ =\ d +
l-n\ =\ d+ \sum_{i=1}^N(i-1)\la_i$.
\qed
\end{lem}

It is clear that the map $v_F$ is an element of $\Vl$.
Thus, the correspondence $F \mapsto v_F$
defines a graded linear
map $\mu : \C[\bs\Sigma] \to \Vl$.

\begin{thm}
\label{thm CC to B}
The kernel of $\mu : \C[\bs\Sigma]
\to \Vl$ is the defining ideal
$I_\CC \subset \C[\bs\Sigma]$ of $\CC$. The map $\mu$ induces a
graded linear isomorphism
$\O_\CC \to \Vl$ of degree $\sum_{i=1}^N(i-1)\la_i$.

\end{thm}

We shall denote this isomorphism by the same letter $\mu$.

\begin{proof} If $F\in I_\CC$, then $v_F =0$ for generic $\bs z$. Hence,
$v_F=0$ as an element of $\Vl$. If $v_F=0$
as an element of $\Vl$, then $F=0$ on a Zariski open subset of
$\CC$. Hence, $F\in I_\CC$. Therefore, ${\rm ker} \,\mu\,=\,I_\CC$.

The graded character of $\O_\CC$ equals the graded character of $\O_\W$
by Corollary \ref{cor iso O to T}. The graded character of $\O_\W$
is given by \Ref{char O}.
The graded character of $\Vl$ is given by
\Ref{char V}. Comparing the characters and using Lemma \ref{lem homog},
we conclude that
the induced map $\mu : \O_\CC \to \Vl$ is an isomorphism.
\end{proof}

\begin{cor}
\label{cor v_1}
Consider the element $v_1\in \Vl$, corresponding to $F=1$ under the isomorphism
$\mu$. Then
$v_1$ is a generator of the one-dimensional graded component of
$\Vl$ of degree $\sum_{i=1}^N(i-1)\la_i$ (compare this $v_1$
with the element $v_1$ in Theorem \ref{first1}).
\qed
\end{cor}

Given an element $F\in \O_\CC$, define a linear map $\nu(F):\Vl \to \Vl$, $v_G
\mapsto v_{FG}$. By Theorem \ref{thm CC to B},
this map is well-defined.

Consider the generators $B^\T_{ij}|_{\CC}$ of $ \O_\CC$
and generators $B^\V_{ij}$ of $\B_\V$, see
Corollary \ref{cor iso O to T} and Section \ref{sec B^V_i}.

\begin{lem}
\label{lem zeta}
For any $(i,j)$, the linear map $\nu(B^\T_{ij}|_{\CC}) : \Vl \to \Vl$,
$v_F \mapsto v_{B^\T_{ij}F}$, coincides with the map $B^\V_{ij}$.
\end{lem}

\begin{proof}
The lemma follows from part (ii) of Theorem \ref{thm X to Vn}.
\end{proof}

\begin{cor}
\label{cor nu-iso}
The map $F\mapsto \nu(F)$ is an algebra isomorphism
$\nu : \O_\CC \to \B_\V$.
\qed
\end{cor}

\begin{cor}
\label{cor cc b}
The maps\/
$\mu : \O_\CC\to \Vl$
and\/
$\nu : \O_\CC \to \B_\V$
intertwine the action of the multiplication
operators on $\O_{\CC}$ and the action of the Bethe algebra $\B_{\V}$
on $(\V^S)^{sing}_{\bs\la}$, that is, for any $F,G\in\O_{\CC}$, we have
\beq
\mu(FG)\,=\,\nu(F)\,\mu(G)\,.
\eeq
\qed
\end{cor}

\begin{cor}
\label{cor cycle}
Consider the element $v_1\in \Vl$, corresponding to $F=1$ under the isomorphism
$\mu$. Let us use this element in the definition of the isomorphism $\eta$ of Theorem
\ref{first1}.
Then the throughout compositions
\bea
\O_\CC \xrightarrow{\mu} \Vl \xrightarrow{\eta^{-1}} \O_\W \xrightarrow{\iota^*}
\O_\CC\ ,
\qquad
\O_\CC \xrightarrow{\nu} \B_\V \xrightarrow{\zeta^{-1}} \O_\W \xrightarrow{\iota^*}
\O_\CC
\eea
are the identity maps.
\qed
\end{cor}

\subsection{Inverse map to $\nu :\O_\CC \to \Vl$}

For $v\in \Vl$, define a function $f_v$
on a Zariski open subset of $\CC$ as follows.
For a generic point $\bs \Sigma\in \CC$, let $\bs T=(\bs z,\bs t)$
be a point of the critical set
$C \subset \C^l$ which projects to $\bs\Sigma$. Let $\omega(\bs z,\bs t)$ be the Bethe
vector corresponding the point $(\bs z,\bs t)$. Set
\bea
f_v(\bs\Sigma)\ = \ S(v(\bs z),\omega(\bs z,\bs t))\ ,
\eea
where $S$ is the tensor Shapovalov form on $V^{\otimes n}$,
cf.~Theorem~\ref{thm X to Vn}.

\begin{thm}
\label{thm inv O_CC}
For any $v\in\Vl$, the scalar
function $f_v$ is the restriction to $\CC$ of a polynomial.
Moreover,
the map $\Vl \to \O_\CC$, $v\mapsto f_v$, is the inverse map to the isomorphism
$\nu :\O_\CC \to \Vl$.
\end{thm}

\begin{proof}
Any element of $\Vl$ has the form of $v_F$ for a suitable
$F\in \C[\bs\Sigma]$, see \Ref{B sum}.
In that case,
\bea
f_{v_F}(\bs z, \bs t)\ = \ S\left( \frac 1{l_1!\dots l_{N-1}!}
\sum_{(\bs z,\bs p )\in C_{\bs z}} \frac {F(\bs z,\bs p)\,\omega(\bs z,\bs p)}
{{\rm Hess}_{\bs t}\log \Phi (\bs z,\bs p)}, \omega(\bs z,\bs t))\right)\ = \
F(\bs z,\bs t)\
\eea
by Theorem \ref{thm X to Vn}. This identity proves the theorem.
\end{proof}

\section{Proof of Theorem \ref{thm conj}}
\label{proofs}

\subsection{ The Shapovalov form and asymptotics of $v_F$}
Let $\bs T^0$ be a point of the critical set $C\subset \C^l$, see
Section \ref{Master function}.

Consider the germ at $0\in\C$ of a generic analytic curve $\C\to\C^l$,
$s\mapsto\bs T(s)=(\bs z(s),\bs t(s))$, with $\bs T(0)=\bs T^0$ such that
for any small nonzero $s$, the point $(\bs z(s),\bs t(s))$ is a nondegenerate
critical point of $\log\Phi(\bs z(s),\,\cdot\,)$, and $\bs z(s)$ has distinct
coordinates.
The corresponding Bethe vector has the form,
$\omega(\bs T(s)) = w_\al s^\alpha + o(s^\alpha)$, where
$\alpha$ is a rational number and
$w_\al \in (V^{\otimes n})^{sing}_{\bs\la}$ is a nonzero vector.

Let $X^0$ denote the point of $\W$ corresponding to $\bs T^0$.
Namely, we take the image $\bs \Sig^0$ of $\bs T^0$ in $\CC$
under the factorization by the $S_{\bs l}$-action
and then
set $X^ 0= \iota(\bs \Sig^0)$.

\begin{lem}
\label{lem Shap and Bethe}
Assume that $X^0$ is not a critical point of the Wronski map $\W \to
\C^n$. Then $S(w_\al,w_\al)$ is a nonzero number, where $S$
is the tensor Shapovalov form.
\end{lem}

\begin{proof} For a small nonzero $s$, the Bethe
vectors corresponding to
$S_{l_1}\times \dots\times S_{l_{N-1}}$-orbits of the
critical points of
$\log \Phi (\bs z(s),\,\cdot\,)$
form a basis of
$(V^{\otimes n})^{sing}_{\bs\la}$,
\ see \cite {MV2}.
That basis is orthogonal with respect to the Shapovalov form.
The Shapovalov form is nondegenerate on $(V^{\otimes n})^{sing}_{\bs\la}$.
By assumptions of the lemma, the limit of the direction of the Bethe vector
$\omega(\bs z(s), \bs t(s))$
as $s\to 0$ is different from the limits of the directions of the
other Bethe vectors of the basis.
These remarks imply the lemma.
\end{proof}

\begin{cor}
\label{cor al}
If $\alpha\leq 0$, then the ratio
$\omega(\bs T(s))/
{\rm Hess}_{\bs t}\log \Phi (\bs T(s))$
has well-defined limit as $s\to 0$.
\end{cor}

\begin{proof}

We have
\be
{\rm Hess}_{\bs t}\log\Phi(\bs T(s))\ =\
S(\omega(\bs T(s)),\omega(\bs T(s)))\ =\
s^{-2\al} S(\omega_\al,\omega_\al) + o(s^{-2\al}) ,
\ee
so the ratio $\omega(\bs T(s))/{\rm Hess}_{\bs t}\log\Phi(\bs T(s))$
has order $s^{-\alpha}$ as $s\to 0$.
\end{proof}

\subsection{Possible places of irregularity of $v_F$}
\label{Possible places of irregularity}
To prove Theorem \ref{thm conj}, we need to show that $v_F$ is
regular outside of at most a codimension-two algebraic subset of
$\C^n$. There are three possible codimension-one irregularity places
of $v_F$:
\plainlabel
\refstepcounter{equation}
\begin{enumerate}
\label{i}
\item[\Ref{i}]
A pole of $v_F$ may occur at a place where $\bs z$ has equal coordinates.
\refstepcounter{equation}
\label{ii}
\item[\Ref{ii}]
A pole of $v_F$ may occur at a place where $\bs z$ has distinct coordinates and
the function \,$\log\Phi(\bs z,\,\cdot\,)$ \,has a degenerate critical point.
\refstepcounter{equation}
\label{iii}
\item[\Ref{iii}]
A pole of $v_F$ may occur at a place where $\bs z$ has distinct coordinates
and there is a critical point which moved to a position
with $t^{(1)}_i=z_j$ for some pair $(i,j)$, or to a position with
$t^{(a)}_i=t^{(a)}_j$ for some triple $(a,i,j)$, \,$a>0$, \,$i\ne j$,
or to a position with $t^{(a)}_i=t^{(a+1)}_j$ for some triple $(a,i,j)$,
\,$a>0$\,.
\end{enumerate}

Problem \Ref{i} is treated in \cite{MV3}. By Lemmas 4.3 and 4.4 of \cite
{MV3}, the map $v_F$ is regular at generic points of the hyperplanes
$z_i=z_j$. (In fact, it is shown in Lemmas 4.3 and 4.4 of \cite{MV3}, that the number
$\alpha$ of Corollary \ref{cor al} is negative at generic points of possible
irregularity corresponding to such hyperplanes, see \cite{MV3}.)

Problem \Ref{ii} of possible irregularity of $v_F$ at
the places, where
$\log \Phi (\bs z,\,\cdot\,)$
has a degenerate critical point, is treated in a standard
way using integral representation \Ref{integral}. One replaces the sum in
\Ref{B sum} by an integral over a cycle which can serve all $\bs z$ that are
close to a given one, and then observes that
the integral is holomorphic in $\bs z$; see, for example,
Sections~5.13, 5.17, 5.18 in~\cite{AGV}.

Thus, to prove Theorem \ref{thm conj} we need to show that generic points of
type \Ref{iii} correspond to the points of $\W$ which are noncritical for
the Wronski map and which have $\al\leq 0$.

\subsection{Flag exponents}

A point $X\in\W$ is an $N$-dimensional space of polynomials with a basis
$g_1(u), \dots, g_N(u)$ such that deg $g_i=\la_i+N-i$.
Each polynomial $g_i$ is defined up to multiplication by a number and
addition of a linear combination of $g_{i+1},\dots,g_N$.

For any $a\in \C$ define distinct integers $\bs
d_{X,a}=(d_1,\dots,d_N)$ called the flag exponents of $X$ as follows.
Choose a basis $g_1,\dots,g_N$ of $X$ (not changing the degrees of
these polynomials) so that $g_1,\dots, g_N$ have different orders at
$u=a$ and set $d_i$ to be the order of $g_i$ at $u=a$.

We say that $X$ is of type $\bs d$ if there exists $a\in\C$ such that
$\bs d_{X,a}=\bs d$. For every $\bs d$, denote by $\W_{\bs d}\subset
\W$ the closure of the subset of points of type $\bs d$. We are
interested in the subsets $\W_{\bs d}\subset \W$ which are of codimension
one and whose points
correspond to Problem \Ref{iii}.
Such subsets will be called {\it essential}.

For example, for $N=2$, the subset $\W_{(0,2)}$ is the only essential subset.
For $N=3$, the only essential subsets are $\W_{(1,3,0)}$, $\W_{(1,0,2)}$ and $\W_{(0,2,1)}$.

\begin{lem}
\label{lem list} For given $N$, if \,
$\W_{\bs d}$ is essential, then $\bs d$ is one of the following $2N-3$ indices,
\bea
\bs d_{1+} &=& (N-2,N,N-3,N-4,\dots,1,0) ,
\\
\bs d_{i+} &=& (N-1,N-2,\dots, N-i+1,N-i-1,N-i-2,N-i,N-i-3,\dots,1,0) ,
\\
\bs d_{i-} &=& (N-1,N-2,\dots, N-i+1,N-i-2,N-i,N-i-1,N-i-3,\dots,1,0)
\eea
for $i=2,\dots,N-1$.
\end{lem}

\begin{proof}
The lemma is proved by straightforward counting of codimensions.
\end{proof}

If $X$ is a point of $\W_{\bs d_{1+}}$, then for a suitable ordering of its
root coordinates we have $z_1=t^{(1)}_1=t^{(1)}_2$.
If $X$ is a point of $\W_{\bs d_{i+}}$, $i>1$, then for a suitable ordering of its
root coordinates we have $t^{(i-1)}_1=t^{(i)}_1=t^{(i)}_2$.
If $X$ is a point of $\W_{\bs d_{i-}}$, then for a suitable ordering of its
root coordinates we have $t^{(i-1)}_1=t^{(i-1)}_2=t^{(i)}_1$.
Each of these properties is a problem of type \Ref{iii}.

\begin{lem}
\label{lem irred}
Each essential subset is irreducible.
\end{lem}

\begin{proof}
It is easy to see that an essential subset is the image of an affine space
under a suitable map.
\end{proof}

\begin{lem}
\label{lem not crit}
Generic points of every essential subset are not critical for
the Wronski map.
\end{lem}

\begin{proof}
The proof is similar to the proof in Proposition 8 of \cite{EG}
of the fact that the Jacobian
$\det \Delta_{\bs q}$ is nonzero.
\end{proof}

\subsection{Proof of Theorem \ref{thm conj} }
\subsubsection{}
\label{subsub intr}
Let $\W_{\bs d}$ be an arbitrary essential subset.
%
We fix a
certain positive integer $q$. Then
for any numbers $\bs r = (r_0,r_1, r_2, \dots, r_q)$, such that $r_0\in\C$, $r_i\in\R$ for $i>0$,
$0<r_1< r_2 < \dots <r_q$,
we choose a point $X_{\bs r}(\ep,s)\in \W$ depending on two parameters $\ep,s$
so that
$X_{\bs r}(\ep,0)\in\W_{\bs d}$ and the point $X_{\bs r}(\ep,s)$ is nice for small nonzero $s$.
The dependence of $X_{\bs r}(\ep,s)$ on $\bs r$ in our construction is generic
in the following sense.
For any hypersurface $\mc Z \subset \W_d$
we can fix $\bs r$ so that the curve $X_{\bs r}(\ep, 0)$ does not lie in $\mc Z$.

For any fixed $\bs r$, we
choose ordered root coordinates $\bs T_{\bs r}(\ep,s)$ of $X_{\bs r}(\ep,s)$ and
consider the corresponding Bethe vector $\omega(\bs T_{\bs r}(\ep,s))$.
We choose a suitable coordinate $\omega_J(\bs T_{\bs r}(\ep,s))$ of the Bethe vector
and show that for small $\ep$
the coordinate $\omega_J(\bs T_{\bs r}(\ep,s))$ has nonzero limit as $s\to 0$.
That statement and Corollary \ref{cor al} show that the corresponding summand in
\Ref{B sum} is regular at $\W_{\bs d}$.

The proof that $\omega_J(\bs T_{\bs r}(\ep,s))$ has nonzero limit is lengthy.
We present it for $N = 2$ and $3$. The proof for $N>3$ is similar.

\subsubsection{Proof for $N=2$}
A point $X\in\W$ is a two-dimensional space of polynomials.
The only essential subset is $\W_{(0,2)}$. This essential subset corresponds
to the problem $z_{\la_1+\la_2}=t^{(1)}_{\la_2-1}=t^{(1)}_{\la_2}$ of
type~\Ref{iii} (after relabeling the root coordinates).

For any numbers $\bs r = (r_0,r_1, r_2, \dots, r_{\la_2+\la_1-1})$,
such that $r_0\in\C$, $r_i\in\R$ for $i>0$,
$0<r_1< r_2 < \dots <r_{\la_2+\la_1-1}$,
we choose $X_{\bs r}(\ep,s)$ to be the two-dimensional space of polynomials
spanned by
\bea
g_2(u) = (u-r_0)^{\la_2} + \sum_{i=2}^{\la_2-1}a_i(u-r_0)^i - s^2a_2\ ,
\qquad
g_1(u) = (u-r_0)^{\la_1+1} + \sum_{i=0}^{\la_1} b_i(u-r_0)^i\ ,
\eea
where $a_{\la_2-1} =\ep^{r_{1}},\ a_{\la_2-i}/a_{\la_2-i+1}= \ep^{r_{i}}$,\
$i=2,\dots,\la_2-2$,\
$b_{\la_1}=\ep^{r_{\la_2-1}},\ b_{\la_1-i}/b_{\la_1-i+1}= \ep^{r_{\la_2+i-1}}$,\ $i=1,\dots,\la_1$.
We have $X_{\bs r}(\ep,0)\in \W_{(0,2)}$.

Clearly, the dependence of $X_{\bs r}(\ep,s)$ on $\bs r$ is generic in the
sense defined in Section \ref{subsub intr}.

\medskip

We consider the asymptotic zone $1 \gg |\ep| \gg |s| > 0$ and
describe the asymptotics in that zone
of the roots of $g_2$ and Wronskian $\Wr(g_1,g_2)$.
The leading terms of asymptotics are obtained by the Newton polygon method.
If the leading term of some root is at least of order $s^2$,
we shall write that this root equals zero.

The roots of $g_2$ have the form:
\bea
t^{(1)}_1\sim r_0-\ep^{r_1}, \ \ t^{(1)}_2\sim r_0-\ep^{r_2},\ \dots\ ,\ \
t^{(1)}_{\la_2-2}\sim r_0-\ep^{r_{\la_2-2}} ,\ \
t^{(1)}_{\la_2-1}\sim r_0+s,\ \ t^{(1)}_{\la_2} \sim r_0-s.
\eea
The Wronskian is a polynomial in $u, \ep$. Below we present only the monomials
corresponding to the line segments of the Newton polygon
important for the leading asymptotics of the roots,
\bea
\Wr(g_1,g_2)
&=&
(\la_1+1-\la_2)(u-r_0)^{\la_2+\la_1} + \sum_{i=2}^{\la_2-1} (\la_1+1-i)a_i(u-r_0)^{\la_1+i} +
\\
&+&
a_{2} \sum_{i=0}^{\la_1} (i-2) b_i(u-r_0)^{i+1} + \dots\ .
\eea
It follows from this formula that the roots of $\Wr(g_1,g_2)$ have the form:
\begin{align*}
& z_1\sim r_0-\frac{\la_1-\la_2+2}{\la_1-\la_2+1}\,\ep^{r_{1}} ,\ \
z_2\sim r_0-\frac{\la_1-\la_2+3}{\la_1-\la_2+2}\,\ep^{r_{2}} ,\ \dots\ ,\
z_{\la_2-2} \sim r_0- \frac{\la_1-1}{\la_1-2}\,\ep^{r_{\la_2-2}} ,
\\[3pt]
& z_{\la_2-1} \sim r_0- \frac{\la_1-2}{\la_1-1}\,\ep^{r_{\la_2-1}} ,\ \dots\ ,\
z_{\la_2+\la_1-4} \sim r_0- \frac{1}{2}\,\ep^{r_{\la_2+\la_1-4}} ,\
\\[4pt]
& z_{\la_2+\la_1-3} \sim r_0+\ep^{(r_{\la_2+\la_1-3}+r_{\la_2+\la_1-2})/2} ,\ \
z_{\la_2+\la_1-2} \sim r_0- \ep^{(r_{\la_2+\la_1-3}+r_{\la_2+\la_1-2})/2} ,\ \
\\[3pt]
&
z_{\la_2+\la_1-1} \sim r_0- 2 \ep^{r_{\la_2+\la_1-1}} ,\ \
z_{\la_2+\la_1} \sim r_0\ .
\end{align*}
The point \,$\bs T_{\bs r}(\ep,s)\,=\,
(z_1,\dots, z_{\la_2+\la_1}, \,t^{(1)}_1,\dots, t^{(1)}_{\la_2})$
\,is a point of root coordinates of $X_{\bs r}(\ep,s)$.

Let us call the root coordinates $t^{(1)}_{\la_2-1}$, $t^{(1)}_{\la_2}$,
$z_{\la_2+\la_1}$ exceptional, and the remaining root coordinates regular.
For each regular root coordinate $y$ the leading term of asymptotics
of $y-r_0$ as $\ep\to 0$ has the form $A\ep^B$ for suitable numbers $A\ne 0$, $B$.

\begin{lem}
\label{lem z neq t N=2}
The pairs $(A,B)$ are different for different regular root coordinates.
\end{lem}

\begin{proof}
A proof is by inspection of the list.
\end{proof}

For each exceptional coordinate $y$ the the absolute value of the
difference $y-r_0$ is
much smaller as $\ep\to 0$ than for any regular coordinate.

\medskip

The Bethe vector is the vector \
$\omega(\bs T_{\bs r}(\ep,s))\,=\,
\sum_J\,\omega_J(\bs T_{\bs r}(\ep,s))\,e_Jv$,\
where the sum is over all admissible $J$, see Section
\ref{Universal weight function and Bethe vectors}.
An
admissible $J=(j_1,\dots,j_{\la_1+\la_2})$ consists of ones and twos
with exactly $\la_2$ twos.
Choose $J$ with $j_i=2$ for $i=1, \dots , \la_2$.
Then
\bean
\label{om N=2}
\omega_J(\bs T_{\bs r}(\ep,s))\ = \
\sum_{\sigma \in S_{\la_2}} \prod_{i=1}^{\la_2}
\frac 1 {t^{(1)}_{\sigma(i)}-z_{i}}\ .
\eean

\begin{lem}
\label{lem limit N=2}
For small $\ep$, the function $\omega_J(\bs T_{\bs r}(\ep,s))$ has well-defined
limit as $s\to 0$.
\end{lem}

\begin{proof}
By Lemma \ref{lem z neq t N=2},
each summand in \Ref{om N=2} has well-defined limit.
\end{proof}

Our goal is to show that
$\bar \omega_J(\ep) = \lim_{s\to 0}\,\omega_J(\bs T_{\bs r}(\ep,s))$
is nonzero for small $\ep$.

\medskip

If $f$ is a function of $\ep$ and $f\sim A(f)\ep^{B(f)}$ for some numbers
$A(f)\ne 0,\,B(f)$
as $\ep\to 0$,
then we call $f$ acceptable,
$B(f)$ the order of $f$ and $A(f)$ the leading coefficient of $f$.
If the absolute value of $f$ is smaller than any positive power of $\ep$ or is the zero function,
then we set $B(f)= \infty$.

\medskip
For every $\sigma$, the limit
$
q_{\sigma}=\lim_{s\to 0}(\prod_{i=1}^{\la_2}
\frac 1 {t^{(1)}_{\sigma(i)}-z_{i}})$ is
an acceptable function of order
$B(q_{\sigma})= - \sum_{i=1}^{\la_2} \min (B(t^{(1)}_{\sigma(i)}-r_0), B(z_{i}-r_0))$.
In particular, \
$B(q_{\sigma})\geq - B(z_{\la_2-1}-r_0) - B(z_{\la_2}-r_0) -
\sum_{i=1}^{\la_2-2} B(t^{(1)}_i-r_0)$.

\begin{lem}
\label{lem N=2}
The function $\bar \omega_J(\ep)$ is acceptable.
Its order and leading coefficient are given by the formulas
\be
B(\bar \omega_J(\ep))\ =\
- B(z_{\la_2-1}-r_0) - B(z_{\la_2}-r_0) -
\sum_{i=1}^{\la_2-2} B(t^{(1)}_i-r_0)\ ,
\ee
\vvn-.2>
\be
A(\bar \omega_J(\ep))\ =\ 2 \,\frac 1{A(z_{\la_2-1}-r_0)A(z_{\la_2}-r_0)}
\prod_{i=1}^{\la_2-2}\frac 1{A(t^{(1)}_{i}-r_0)-A(z_{i}-r_0)}\ .
\ee
\qed
\end{lem}

By Lemma \ref{lem N=2}, $\bar \omega_J(\ep)$ is nonzero for small $\ep$
and therefore, $\al\leq 0$ for generic points of $\W_{(0,2)}$.
Theorem \ref{thm conj} is proved for $N=2$.

\subsubsection{Proof for $N=3$ and $\W_{(1,3,0)}$ }

A point $X\in\W$ is a three-dimensional space of polynomials.
We study the problem
$z_{\la_1+\la_2+\la_3}=t^{(1)}_{\la_2+\la_3-1}=t^{(1)}_{\la_2+\la_3}$ of
type~\Ref{iii} (after relabeling the root coordinates).

For any numbers $\bs r = (r_0,r_1, r_2, \dots, r_{\la_3+\la_2+\la_1-1})$,
such that $r_0\in\C$, $r_i\in\R$ for $i>0$,
$0<r_1< r_2 < \dots <r_{\la_3+\la_2+\la_1-1}$,
we choose $X_{\bs r}(\ep,s)$ to be the three-dimensional space of polynomials
spanned by
\begin{align*}
g_3(u)\,&{}=\,(u-r_0)^{\la_3} + \sum_{i=0}^{\la_3-1} a_i(u-r_0)^i\ ,
\\[4pt]
g_2(u)\,&{}=\,(u-r_0)^{\la_2+1} + \sum_{i=3}^{\la_2} b_i(u-r_0)^i -3s^2b_3(u-r_0)\ ,
\\[4pt]
g_1(u)\,&{}=\,(u-r_0)^{\la_1+2} + \sum_{i=1}^{\la_1+1} c_i(u-r_0)^i\ ,
\end{align*}
where $a_{\la_3-1}=\ep^{r_{1}},\ a_{\la_3-i}/a_{\la_3-i+1}= \ep^{r_{i}}$,\ $i=2,\dots,\la_3$,\
$b_{\la_2}=\ep^{r_{\la_3+1}},\ b_{\la_2-i}/b_{\la_2-i+1}= \ep^{r_{\la_3+i+1}}$,\ $i=1,\dots,\la_2-3$,\
$c_{\la_1+1}=\ep^{r_{\la_3+\la_2-1}},\ c_{\la_1-i}/c_{\la_1-i+1}= \ep^{r_{\la_3+\la_2+i}}$,\ $i=0,\dots,\la_1-1$.\
We have $X_{\bs r}(\ep,0)\in \W_{(1,3,0)}$.

Clearly, the dependence of $X_{\bs r}(\ep,s)$ on $\bs r$ is generic
in the sense defined in Section \ref{subsub intr}.

\medskip

We consider the asymptotic zone $1 \gg |\ep| \gg |s| > 0$ and
describe the asymptotics in that zone
of the roots of the polynomials $g_3$, $\Wr(g_2,g_3)$, $\Wr(g_1,g_2,g_3)$.
We obtain the leading terms of asymptotics by the Newton polygon method.
If the leading term of some root is at least of order $s^2$,
we shall write that this root equals zero.

The roots of $g_3$ are of the form:
\bea
t^{(2)}_1 \sim r_0-\ep^{r_{1}}, \ \
t^{(2)}_2\sim r_0-\ep^{r_{2}}, \ \dots\ ,\ \
t^{(2)}_{\la_3}\sim r_0-\ep^{r_{\la_3}}\ .
\eea
We have
\bea
\Wr(g_2,g_3)\,
&=&
\,(\la_2+1-\la_3)(u-r_0)^{\la_2+\la_3} +
\sum_{i=0}^{\la_3-1}
(\la_2+1-i)a_i (u-r_0)^{\la_2+i-1}
\\
&+&
a_0\sum_{i=3}^{\la_2}
ib_i (u-r_0)^{i-1} -3s^2a_0b_3 + \dots\ ,
\eea
where the dots denote the
monomials which are not important for the leading asymptotics of
the roots.
The roots of $\Wr(g_2,g_3)$ are of the form
\begin{align*}
& t^{(1)}_1\sim r_0- \frac{\la_2-\la_3+2}{\la_2-\la_3+1}\,\ep^{r_{1}},\ \dots\ ,\
t^{(1)}_{\la_3}\sim r_0- \frac{\la_2+1}{\la_2}\,\ep^{r_{\la_3}},
\\
& t^{(1)}_{\la_3+1}\sim r_0 - \frac{\la_2}{\la_2+1}\,\ep^{r_{\la_3+1}},\ \dots\ ,\
t^{(1)}_{\la_3+\la_2-2}\sim r_0 - \frac{3}{4}\,\ep^{r_{\la_3+\la_2-2}},\ \
\\
&
t^{(1)}_{\la_3+\la_2-1}\sim r_0 + s,\ \
t^{(1)}_{\la_3+\la_2}\sim r_0 - s.
\end{align*}
We have
\begin{align*}
\Wr(g_1,g_2,g_3)\,&{}=\,
(\la_1+1-\la_2)(\la_1+2-\la_3)(\la_2+1-\la_3)(u-r_0)^{\la_3+\la_2+\la_1} +{}
\\
&{}+\,
\sum_{i=0}^{\la_3-1}
(\la_1+1-\la_2)(\la_1+2-i)(\la_2+1-i)a_i(u-r_0)^{i+\la_2+\la_1} +{}
\\
&{}+\,a_{0} \sum_{i=3}^{\la_2}
(\la_1+2-i)(\la_1+2)ib_i(u-r_0)^{\la_1+i-1} +
\\
&{}+\,
a_{0} b_{3}\sum_{i=1}^{\la_1+1} 3i(i-3)c_i(u-r_0)^{i} + \dots\ .
\end{align*}
The roots of $\Wr(g_1,g_2,g_3)$ are of the form
\begin{align*}
& z_1 \sim r_0-\frac{(\la_1+3-\la_3)(\la_2+2-\la_3)}
{(\la_1+2-\la_3)(\la_2+1-\la_3)}\,\ep^{r_{1}},\ \dots\ ,\
z_{\la_3} \sim r_0-\frac {(\la_1+2)(\la_2+1)}{(\la_1+1)\la_2}\,\ep^{r_{\la_3}},
\\[3pt]
& z_{\la_3+1} \sim r_0
-\frac {(\la_1+2-\la_2)\la_2}{(\la_1+1-\la_2)(\la_2+1)}\,\ep^{r_{\la_3+1}},
\ \dots\ ,\
z_{\la_3+\la_2-2} \sim r_0-\frac {3(\la_1-1)}{4(\la_1-2)}\,\ep^{r_{\la_3+\la_2-2}},
\\[3pt]
& z_{\la_3+\la_2-1} \sim r_0
-\frac {(\la_1+1)(\la_1-2)}{(\la_1+2)(\la_1-1)}\,\ep^{r_{\la_3+\la_2-1}},\ \dots
\ ,\ z_{\la_3+\la_2+\la_1-4} \sim r_0-\frac{2}{5}\,\ep^{r_{\la_3+\la_2+\la_1-4}},\
\\[3pt]
& z_{\la_3+\la_2+\la_1-3} \sim r_0+\frac{1}{\sqrt2}\,\ep^{(r_{\la_3+\la_2+\la_1-3}+r_{\la_3+\la_2+\la_1-2})/2},\
\
\\[3pt]
& z_{\la_3+\la_2+\la_1-2} \sim r_0-\frac{1}{\sqrt2}\,\ep^{(r_{\la_3+\la_2+\la_1-3}+r_{\la_3+\la_2+\la_1-2})/2},
\\[3pt]
& z_{\la_3+\la_2+\la_1-1} \sim r_0-\ep^{r_{\la_3+\la_2+\la_1-1}},\ \
z_{\la_3+\la_2+\la_1} \sim r_0\,.
\end{align*}
The point \,${\bs T_{\bs r}(\ep,s)\,=\,(z_1,\dots, z_{\la_3+\la_2+\la_1},\,
t^{(1)}_1,\dots,t^{(1)}_{\la_3+\la_2},\,t^{(2)}_1,\dots,t^{(2)}_{\la_3})}$
\,is a point of root coor\-dinates of $X_{\bs r}(\ep,s)$.

Let us call the root coordinates $t^{(1)}_{\la_3+\la_2-1}$,
$t^{(1)}_{\la_3+\la_2}$, $z_{\la_3+\la_2+\la_1}$ exceptional, and the remaining
root coordinates regular. For each regular root coordinate $y$ the leading
term of asymptotics of $y-r_0$ as $\ep\to 0$ has the form $A\ep^B$ for suitable numbers
$A\ne 0$, $B$.

\begin{lem}
\label{lem z neq t 130}
The pairs $(A,B)$ are different for different regular root coordinates.
\end{lem}

\begin{proof}
A proof is by inspection of the list.
\end{proof}

For each exceptional coordinate $y$ the absolute value of the
difference $y-r_0$ is
much smaller as $\ep\to 0$ than for any regular coordinate.

The Bethe vector has the form
$\omega(\bs T_{\bs r}(\ep,s))\,=\,
\sum_J\,\omega_J(\bs T_{\bs r}(\ep,s))\,e_Jv$,\
where the sum is over all admissible $J$, see Section
\ref{Universal weight function and Bethe vectors}. An
admissible $J=(j_1,\dots,j_{\la_3+\la_2+\la_1})$ consists of ones, twos and
threes with exactly $\la_3$ threes and $\la_2$ twos.
Choose $J$ with $j_i=3$ for $i=1,\dots, \la_3$
and $j_i=2$ for $i=\la_3+\la_2-1, \la_3+\la_2, \dots , \la_3+2\la_2-2$.
Then $\omega_J(\bs T_{\bs r}(\ep,s))$ is given by the formula
\bean
\label{130}
\phantom{aaa}
\omega_J(\bs T_{\bs r}(\ep,s))\ = \
\sum_{\sigma \in S_{\la_3+\la_2}}
\sum_{\tau \in S_{\la_3}}
\prod_{i=1}^{\la_3}
\frac 1
{(t^{(2)}_{\tau(i)}-t^{(1)}_{\sigma(i)})
(t^{(1)}_{\sigma(i)}-z_{i})}
\prod_{i=\la_3+1}^{\la_3+\la_2}
\frac 1 {t^{(1)}_{\sigma(i)}-z_{\la_2+i-2}}.
\eean

\begin{lem}
\label{lem limit 130}
For small $\ep$, the function $\omega_J(\bs T_{\bs r}(\ep,s))$ has well-defined limit
as $s\to 0$.
\end{lem}

\begin{proof}
By Lemma \ref{lem z neq t 130},
each summand in \Ref{130} has well-defined limit.
\end{proof}

Our goal is to show that
$\bar \omega_J(\ep) = \lim_{s\to 0}\,\omega_J(\bs T_{\bs r}(\ep,s))$
is nonzero for small $\ep$.

\medskip

For every $\sigma$ the second product in \Ref{130},
has well-defined limit\\
$q_{\sigma}=\lim_{s\to 0}
(\prod_{i=\la_3+1}^{\la_3+\la_2}
\frac 1 {t^{(1)}_{\sigma(i)}-z_{\la_2+i-2}})$.
That limit is an acceptable function of order
$B(q_{\sigma})=
-\sum_{i=\la_3+1}^{\la_3+\la_2}\min(B(t^{(1)}_{\sigma(i)}-r_0),B(z_{\la_2+i-2}-r_0))$.
In particular,
\be
B(q_{\sigma})\geq
- B(z_{\la_3+2\la_2-3}-r_0) - B(z_{\la_3+2\la_2-2}-r_0) -
\sum_{i=\la_3+1}^{\la_3+\la_2-2} B(t^{(1)}_i-r_0) .
\ee
The largest second products are those with
\beq
\label{largest 1}
B(q_{\sigma})\ =\ - B(z_{\la_3+2\la_2-3}-r_0) - B(z_{\la_3+2\la_2-2}-r_0) -
\sum_{i=\la_3+1}^{\la_3+\la_2-2} B(t^{(1)}_i-r_0)\ .
\eeq
For every $\sigma, \tau$, the first product in \Ref{130}
has well-defined limit
\\
$p_{\sigma\tau}=\lim_{s\to 0}(\prod_{i=1}^{\la_3}
\frac 1{(t^{(2)}_{\tau(i)}-t^{(1)}_{\sigma(i)})
(t^{(1)}_{\sigma(i)}-z_{i})})$.
That limit is an acceptable function of order
\be
B(p_{\sigma\tau})\,=\,- \sum_{i=1}^{\la_3}
(\min (B(t^{(2)}_{\tau(i)}-r_0), B(t^{(1)}_{\sigma(i)}-r_0))+
\min (B(t^{(1)}_{\sigma(i)}-r_0), B(z_{i}-r_0))).
\ee
In particular,
$B(p_{\sigma\tau}) \geq - \sum_{i=1}^{\la_3}(B(t^{(2)}_{i}-r_0)+ B(z_{i}-r_0))$.
The largest first products are those with
\beq
\label{largest nonres 2 130}
B(p_{\sigma\tau})\,=\,
- \sum_{i=1}^{\la_3}
(B(t^{(2)}_{i}-r_0)+ B(z_{i}-r_0)).
\eeq

\begin{lem}
\label{lem 130}
The function $\bar \omega_J(\ep)$ is acceptable.
Its order and leading coefficient are given by the formulas
\bea
B(\bar \omega_J(\ep))\,=\,
&-& \sum_{i=1}^{\la_3} (B(t^{(2)}_{i}-r_0)+ B(z_{i}-r_0))
- B(z_{\la_3+2\la_2-3}-r_0) - B(z_{\la_3+2\la_2-2}-r_0)
\\
&-&
\sum_{i=\la_3+1}^{\la_3+\la_2-2} B(t^{(1)}_i-r_0)\ ,
\eea
\vv-.2>
\begin{align*}
A(\bar \omega_J(\ep))\,=\,
&
2 \,(\la_2-2)! \,
\frac 1{A(z_{\la_3+2\la_2-3}-r_0)A(z_{\la_3+2\la_2-2}-r_0)}\times{}
\\[3pt]
{}\times{}& \prod_{i=1}^{\la_3}
\frac 1
{(A(t^{(2)}_{i}-r_0)-A(t^{(1)}_{i}-r_0))
(A(t^{(1)}_{i}-r_0)-A(z_{i}-r_0))}
\prod_{i=\la_3+1}^{\la_3+\la_2-2} \frac 1{A(t^{(1)}_{i}-r_0)}\ .
\end{align*}
\end{lem}

\begin{proof}
It is easy to see that if $\sigma, \tau$ are such that the second product
in \Ref{130} has order $-B(z_{\la_3+2\la_2-3}-r_0) - B(z_{\la_3+2\la_2-2}-r_0) -
\sum_{i=\la_3+1}^{\la_3+\la_2-2} B(t^{(1)}_i-r_0)$, \ then the first product
has order
$ - \sum_{i=1}^{\la_3} (B(t^{(2)}_{i}-r_0)+ B(z_{i}-r_0))$ only if it equals
$\prod_{i=1}^{\la_3}\frac 1{(t^{(2)}_{i}-t^{(1)}_{i})(t^{(1)}_{i}-z_{i})}$.
This implies the lemma.
\end{proof}

\subsubsection{Proof for $N=3$ and $\W_{(0,2,1)}$ }

We study the problem
$t^{(1)}_{\la_2+\la_3-1}=t^{(1)}_{\la_2+\la_3}=t^{(2)}_{\la_3}$ of
type~\Ref{iii} (after relabeling the root coordinates).

For any numbers $\bs r = (r_0,r_1, r_2, \dots, r_{\la_3+\la_2+\la_1})$,
such that $r_0\in\C$, $r_i\in\R$ for $i>0$,
$0<r_1< r_2 < \dots <r_{\la_3+\la_2+\la_1}$,
we choose $X_{\bs r}(\ep,s)$ to be the three-dimensional space of polynomials
spanned by
\begin{align*}
g_3(u)\,&{}=\,(u-r_0)^{\la_3} + \sum_{i=1}^{\la_3-1} a_i(u-r_0)^i\ ,
\qquad
g_2(u)\,=\,(u-r_0)^{\la_2+1} + \sum_{i=2}^{\la_2} b_i(u-r_0)^i + s^2b_2\ ,
\\
g_1(u)\,&{}=\,(u-r_0)^{\la_1+2} + \sum_{i=0}^{\la_1+1} c_i(u-r_0)^i\ ,
\end{align*}
where $a_{\la_3-1}=\ep^{r_{1}},\ a_{\la_3-i}/a_{\la_3-i+1}= \ep^{r_{i}}$,\ $i=2,\dots,\la_3-1$,\
$b_{\la_2}=\ep^{r_{\la_3}},\ b_{\la_2-i}/b_{\la_2-i+1}= \ep^{r_{\la_3+i}}$,\ $i=1,\dots,\la_2-2$,\
$c_{\la_1+1}=\ep^{r_{\la_3+\la_2-1}},\ c_{\la_1-i}/c_{\la_1-i+1}= \ep^{r_{\la_3+\la_2 +i}}$,\ $i=0,\dots,\la_1$.\
We have $X_{\bs r}(\ep,0)\in \W_{(0,2,1)}$.

Clearly, the dependence of $X_{\bs r}(\ep,s)$ on $\bs r$ is
generic in the sense defined in Section \ref{subsub intr}.

\medskip

We consider the same asymptotic zone $1 \gg |\ep| \gg |s| > 0$.

The roots of $g_3$ are of the form:
\be
t^{(2)}_1 \sim r_0-\ep^{r_{1}}, \ \
t^{(2)}_2\sim r_0-\ep^{r_{2}}, \ \dots\ ,\
t^{(2)}_{\la_3-1}\sim r_0-\ep^{r_{\la_3-1}},\ \
t^{(2)}_{\la_3} = r_0\ .
\ee
We have
\bea
\Wr(g_2,g_3)
&=&
(\la_2+1-\la_3)(u-r_0)^{\la_3+\la_2} +
\sum_{i=1}^{\la_3-1}(\la_2+1-i)a_i (u-r_0)^{\la_2+i}
+
\\
&+&
a_1\sum_{i=2}^{\la_2}(i-1)b_i (u-r_0)^{i} - s^2a_1b_2 + \dots\ .
\eea
The roots of $\Wr(g_2,g_3)$ are of the form
\begin{align*}
& t^{(2)}_1\sim r_0- \frac{\la_2-\la_3+2}{\la_2-\la_3+1}\,\ep^{r_{1}},\ \dots\ ,\
t^{(2)}_{\la_3-1}\sim r_0- \frac{\la_2}{\la_2-1}\,\ep^{r_{\la_3-1}},
\\[3pt]
& t^{(2)}_{\la_3}\sim r_0- \frac{\la_2-1}{\la_2}\,\ep^{r_{\la_3}},\ \dots\ ,\
t^{(2)}_{\la_3+\la_2-2}\sim r_0- \frac{1}{2}\,\ep^{r_{\la_3+\la_2-2}},\ \
t^{(2)}_{\la_3+\la_2-1}\sim r_0+s,\ \
t^{(2)}_{\la_3+\la_2}\sim r_0-s.
\end{align*}
We have
\begin{align*}
\Wr(g_1,g_2 &{},g_3)\,=\,
(\la_1+1-\la_2)(\la_1+2-\la_3)(\la_2+1-\la_3)(u-r_0)^{\la_3+\la_2+\la_1} +{}
\\[4pt]
&{}+\,\sum_{i=1}^{\la_3-1}
(\la_1+1-\la_2)(\la_1+2-i)(\la_2+1-i)a_i(u-r_0)^{i+\la_2+\la_1}+{}
\\
&{}+\,a_1 \sum_{i=2}^{\la_2}(\la_1+2-i)(\la_1+1)(i-1)b_i(u-r_0)^{\la_1+i} +
\\
&{}+\,
a_{1} b_{2}\sum_{i=0}^{\la_1+1} (i-2)(i-1)c_i(u-r_0)^{i} + \dots\ .
\end{align*}
The roots of $\Wr(g_1,g_2,g_3)$ are of the form
\begin{align*}
& z_1 \sim r_0-\frac{(\la_1+3-\la_3)(\la_2+2-\la_3)}
{(\la_1+2-\la_3)(\la_2+1-\la_3)}\,\ep^{r_1},\ \dots\ ,
z_{\la_3-1} \sim r_0-\frac{(\la_1+1)\la_2}{\la_1(\la_2-1)}\,\ep^{r_{\la_3-1}},
\\[3pt]
& z_{\la_3} \sim
r_0 - \frac{(\la_1+2-\la_2)(\la_2-1)}{(\la_1+1-\la_2)\la_2}\,\ep^{r_{\la_3}},
\ \dots\ ,\
z_{\la_3+\la_2-3}\sim r_0-\frac{\la_1}{2(\la_1+1)}\,\ep^{r_{\la_3+\la_2-3}},
\\[3pt]
&
z_{\la_3+\la_2-2}\sim r_0-\frac{(\la_1-1)}{(\la_1+1)}\,\ep^{r_{\la_3+\la_2-2}},
\ \dots\ ,\
z_{\la_3+\la_2+\la_1-3} \sim r_0-\frac13\,\ep^{r_{\la_3+\la_2+\la_1-3}},
\\[3pt]
&
z_{\la_3+\la_2+\la_1-2} \sim r_0+x_1\ep^m,\ \
z_{\la_3+\la_2+\la_1-1} \sim r_0+x_2\ep^m,\ \
z_{\la_3+\la_2+\la_1} \sim r_0+x_3\ep^m\ ,
\end{align*}
where $x_1,x_2,x_3$ are distinct roots of the equation $x^3+1=0$ and
$m = (r_{\la_3+\la_2+\la_1-2} + r_{\la_3+\la_2+\la_1-1}+r_{\la_3+\la_2+\la_1})/3$.

The point
$
\bs T_{\bs r}(\ep,s)\,=\,(z_1,\dots, z_{\la_3+\la_2+\la_1},\,
t^{(1)}_1,\dots,t^{(1)}_{\la_3+\la_2},\,t^{(2)}_1,\dots,t^{(2)}_{\la_3})
$
is a point of root coor\-dinates of $X_{\bs r}(\ep,s)$.

Let us call the root coordinates $t^{(2)}_{\la_3}$, $t^{(1)}_{\la_3+\la_2-1}$,
$t^{(1)}_{\la_3+\la_2}$ exceptional, and the remaining root coordinates
regular.
For each regular root coordinate $y$ the leading
term of asymptotics of $y-r_0$ as $\ep\to 0$ has the form $A\ep^B$ for suitable numbers
$A\ne 0$, $B$.

\begin{lem}
\label{lem z neq t 021n}
The pairs $(A,B)$ are different for different regular root coordinates.
\end{lem}

\begin{proof}
A proof is by inspection of the list.
\end{proof}

For each exceptional coordinate $y$ the the absolute value of the
difference $y-r_0$ is
much smaller as $\ep\to 0$ than for any regular coordinate.

The Bethe vector has the form
$\omega(\bs T_{\bs r}(\ep,s))\,=\,
\sum_J\,\omega_J(\bs T_{\bs r}(\ep,s))\,e_Jv$,\
where the sum is over all admissible $J$, see Section
\ref{Universal weight function and Bethe vectors}. An
admissible $J=(j_1,\dots,j_{\la_3+\la_2+\la_1})$ consists of ones, twos and
threes with exactly $\la_3$ threes and $\la_2$ twos.
Choose $J$ with $j_i=3$ for $i=1,2,\dots, \la_3-1$, $\la_3+\la_2-2$
and $j_i=2$ for $i=\la_3, \la_3+1,\dots, \la_3+\la_2-3, \la_3+\la_2-1,\la_3+\la_2$.
Then $\omega_J(\bs T_{\bs r}(\ep,s))$ is given by the formula
\begin{align}
\label{021n}
\phantom{aaa}
\omega_J(\bs T_{\bs r}(\ep,s))\ =\ \sum_{\sigma \in S_{\la_3+\la_2}}
\sum_{\tau \in S_{\la_3}}\prod_{i=1}^{\la_3-1}
\frac 1{(t^{(2)}_{\tau(i)}-t^{(1)}_{\sigma(i)})
(t^{(1)}_{\sigma(i)}-z_{i})}
& \prod_{i=\la_3+1}^{\la_3+\la_2-2}
\frac 1 {t^{(1)}_{\sigma(i)}-z_{i-1}}\times{}
\\
{}\times\,\frac 1
{(t^{(2)}_{\tau(\la_3)}-t^{(1)}_{\sigma(\la_3)})
(t^{(1)}_{\sigma(\la_3)}-z_{\la_3+\la_2-2})}
& \prod_{i=\la_3+\la_2-1}^{\la_3+\la_2}
\frac 1 {t^{(1)}_{\sigma(i)}-z_{i}}\ .
\notag
\end{align}

\begin{lem}
\label{lem limit 021n}
For small $\ep$, the function $\omega_J(\bs T_{\bs r}(\ep,s))$ has well-defined limit as
$s\to 0$.
\end{lem}

\begin{proof}
Divergent summands in \Ref{021n} are the summands with factors
$\frac 1{t^{(2)}_{\la_3}-t^{(1)}_{\la_3+\la_2-1}}$ or
$\frac 1{t^{(2)}_{\la_3}-t^{(1)}_{\la_3+\la_2}}$. The divergent summands come in pairs.
There are two types of divergent pairs. The first type has the form
\begin{align*}
& p^{Ckij}_1\,=\,
\frac {C}{(t^{(2)}_{\la_3}-t^{(1)}_{\la_3+\la_2-1})
(t^{(1)}_{\la_3+\la_2-1}-z_i)(t^{(2)}_{k}-t^{(1)}_{\la_3+\la_2})
(t^{(1)}_{\la_3+\la_2}-z_j)}\ ,
\\[4pt]
& p^{Ckij}_2\,=\,
\frac {C}{(t^{(2)}_{\la_3}-t^{(1)}_{\la_3+\la_2})
(t^{(1)}_{\la_3+\la_2}-z_i)(t^{(2)}_{k}-t^{(1)}_{\la_3+\la_2-1})
(t^{(1)}_{\la_3+\la_2-1}-z_j)}\ ,
\end{align*}
where $C$ is a common factor. The second type has the form
\begin{align*}
& q^{Cij}_1\,=\,\frac {C}{(t^{(2)}_{\la_3}-t^{(1)}_{\la_3+\la_2-1})
(t^{(1)}_{\la_3+\la_2-1}-z_i)
(t^{(1)}_{\la_3+\la_2}-z_j)}\ ,
\\[4pt]
& q^{Cij}_2\,=\,\frac {C}{(t^{(2)}_{\la_3}-t^{(1)}_{\la_3+\la_2})
(t^{(1)}_{\la_3+\la_2}-z_i)
(t^{(1)}_{\la_3+\la_2-1}-z_j)}\ ,
\end{align*}
where $C$ is a common factor. Each pair has well-defined limit as $s\to 0$,
\bea
\lim_{s\to 0}\
(p^{Ckij}_1+p^{Ckij}_2)\,&=&\,
\lim_{s\to 0}\
\frac {C}{(t^{(2)}_{k}-r_0)(z_i-r_0)(z_j-r_0)}\,(\frac 2{t^{(2)}_k-r_0}+\frac 2{z_j-r_0}
-\frac 2{z_i-r_0})\ ,
\\
\lim_{s\to 0}\ (q^{Cij}_1+q^{Cij}_2)\,&=&\,
\lim_{s\to 0}\
\frac {C}{(z_i-r_0)(z_j-r_0)}\,(\frac 2{z_j-r_0}-\frac 2{z_i-r_0})\ .
\eea
These limits will be called resonant pairs.
\end{proof}

It is easy to see that
$\bar \omega_J(\ep) = \lim_{s\to 0}\,\omega_J(\bs T_{\bs r}(\ep,s))$
is an acceptable function and its order equals
$b= -B(z_{\la_3+\la_2}-r_0)-\sum_{i=1}^{\la_3-1}B(t^{(2)}_i-r_0)-
\sum_{i=1}^{\la_3+\la_2}B(z_i-r_0)$. Indeed, the order of the limit of any convergent summand
in \Ref{021n} is greater than $b$.
The order of any resonant pair is not less than $b$.
There is exactly one resonant pair of order $b$.
That pair is of the second type and corresponds to
\begin{align*}
q_1\ =\ \prod_{i=1}^{\la_3-1} {}&
\frac 1{(t^{(2)}_{i}-t^{(1)}_{i})
(t^{(1)}_{i}-z_{i})}
\prod_{i=\la_3}^{\la_3+\la_2-3}
\frac 1 {t^{(1)}_{i}-z_{i}}\times{}
\\
&\phantom{(t^{(2)}_{i}}\times
\frac 1{(0-s)(s-z_{\la_3+\la_2-2})
(t^{(1)}_{\la_3+\la_2-2}-z_{\la_3+\la_2-1})(-s-z_{\la_3+\la_2})}\ ,
\\[8pt]
q_2\ =\ \prod_{i=1}^{\la_3-1} {}&
\frac 1{(t^{(2)}_{i}-t^{(1)}_{i})(t^{(1)}_{i}-z_{i})}
\prod_{i=\la_3}^{\la_3+\la_2-3}\frac 1 {t^{(1)}_{i}-z_{i}}\times{}
\\
&\phantom{(t^{(2)}_{i}}\times
\frac 1{(0+s)(-s-z_{\la_3+\la_2-2})
(t^{(1)}_{\la_3+\la_2-2}-z_{\la_3+\la_2-1})(s-z_{\la_3+\la_2})}\ .
\end{align*}
Thus, $\bar \omega_J(\ep)$ is nonzero for small $\ep$.

\subsubsection{Proof for $N=3$ and $\W_{(1,0,2)}$ }
We study the problem
$t^{(1)}_{\la_2+\la_3}=t^{(2)}_{\la_3-1}=t^{(2)}_{\la_3}$ of type~\Ref{iii}
(after relabeling the root coordinates).

For any numbers $\bs r = (r_0,r_1, r_2, \dots, r_{\la_3+\la_2+\la_1})$,
such that $r_0\in\C$, $r_i\in\R$ for $i>0$,
$0<r_1< r_2 < \dots <r_{\la_3+\la_2+\la_1}$,
we choose $X_{\bs r}(\ep,s)\in \W$ to be the three-dimensional space of polynomials
spanned by
\bea
g_3(u) &=& (u-r_0)^{\la_3} + \sum_{i=2}^{\la_3-1} a_i(u-r_0)^i - a_2s^2\ ,
\qquad
g_2(u) = (u-r_0)^{\la_2+1} + \sum_{i=0}^{\la_2} b_i(u-r_0)^i \ ,
\\
g_1(u) &=& (u-r_0)^{\la_1+2} + \sum_{i=1}^{\la_1+1} c_i(u-r_0)^i\ ,
\eea
where $a_{\la_3-1}=\ep^{r_1},\ a_{\la_3-i}/a_{\la_3-i+1}= \ep^{r_i}$,\ $i=2,\dots,\la_3-2$,\
$b_{\la_2}=\ep^{r_{\la_3-1}},\ b_{\la_2-i}/b_{\la_2-i+1}= \ep^{r_{\la_3+i-1}}$,\ $i=1,\dots,\la_2$,\
$c_{\la_1+1}=\ep^{r_{\la_3+\la_2}},\ c_{\la_1-i}/c_{\la_1-i+1}= \ep^{r_{\la_3+\la_2 +i+1}}$,\ $i=0,\dots,\la_1-1$.\
We have $X_{\bs r}(\ep,0)\in \W_{(1,0,2)}$.

Clearly, the dependence of $X_{\bs r}(\ep,s)$ on $\bs r$ is such that the corresponding curve
$X_{\bs r}(\ep,0)$ is generic in $\W_{\bs d}$ in the sense defined
in Section \ref{subsub intr}.

\medskip

We consider the same asymptotic zone $1 \gg |\ep| \gg |s| > 0$.

The roots of $g_3$ are of the form:
\be
t^{(2)}_1 \sim r_0-\ep^{r_1}, \ \dots\ ,\
t^{(2)}_{\la_3-2}\sim r_0-\ep^{r_{\la_3-2}}\ ,\ \
t^{(2)}_{\la_3-1}\sim r_0+s,\ \
t^{(1)}_{\la_3}\sim r_0-s.
\vv-.2>
\ee
We have
\vvn-.5>
\bea
\Wr(g_2,g_3)
&= &
(\la_2+1-\la_3)(u-r_0)^{\la_3+\la_2} +
\sum_{i=2}^{\la_3-1}
(\la_2+1-i)a_i (u-r_0)^{\la_3+i-1}
+
\\
&+&
a_2\sum_{i=0}^{\la_2}
(i-2)b_i (u-r_0)^{i+1} + \dots\ .
\eea
The roots of $\Wr(g_2,g_3)$ are of the form
\begin{align*}
& t^{(1)}_1\sim r_0- \frac{\la_2-\la_3+2}{\la_2-\la_3+1}\,\ep^{r_1},\ \dots\ ,\
t^{(1)}_{\la_3-2}\sim r_0- \frac{\la_2-1}{\la_2-2}\,\ep^{r_{\la_3-2}},
\\[3pt]
& t^{(1)}_{\la_3-1}\sim r_0- \frac{\la_2-2}{\la_2-1}\,\ep^{r_{\la_3-1}},\ \dots\ ,\
t^{(1)}_{\la_3+\la_2-4}\sim r_0-\frac12\,\ep^{r_{\la_3+\la_2-4}},
\\[3pt]
& t^{(1)}_{\la_3+\la_2-3}\sim \ep^{(r_{\la_3+\la_2-3}+r_{\la_3+\la_2-2})/2},\ \
t^{(1)}_{\la_3+\la_2-2}\sim -\ep^{(r_{\la_3+\la_2-3}+r_{\la_3+\la_2-2})/2},\ \
\\[3pt]
&
t^{(1)}_{\la_3+\la_2-1}\sim r_0- 2\ep^{r_{\la_3+\la_2-1}},\ \
t^{(1)}_{\la_3+\la_2}\sim r_0\ .
\end{align*}
We have
\begin{align*}
\Wr(g_1,g_2,g_3)\,&{}=\,
(\la_1+1-\la_2)(\la_1+2-\la_3)(\la_2+1-\la_3)(u-r_0)^{\la_3+\la_2+\la_1} +{}
\\[3pt]
&{}+\,\sum_{i=2}^{\la_3-1}
(\la_1+1-\la_2)(\la_1+2-i)(\la_2+1-i)a_i(u-r_0)^{i+\la_2+\la_1} +{}
\\
&{}+\,a_{2} \sum_{i=0}^{\la_2}
(\la_1+2-i)\la_1(i-2)b_i(u-r_0)^{\la_1+i+1} -
\\
&{}
-\,a_{2} b_{0}\,\sum_{i=1}^{\la_1+1} 2i(i-2)c_i(u-r_0)^{i-1} + \dots\ .
\end{align*}
The roots of $\Wr(g_1,g_2,g_3)$ are of the form
\begin{align*}
& z_1 \sim r_0-\frac{(\la_1+3-\la_3)(\la_2+2-\la_3)}
{(\la_1+2-\la_3)(\la_2+1-\la_3)}\,\ep^{r_1},\ \dots\ ,\
z_{\la_3-2} \sim r_0-\frac {\la_1(\la_2-1)}{(\la_1-1)(\la_2-2)}\,\ep^{r_{\la_3-2}},
\\[3pt]
& z_{\la_3-1} \sim r_0
-\frac {(\la_1+2-\la_2)(\la_2-2)}{(\la_1+1-\la_2)(\la_2-1)}\,\ep^{r_{\la_3+1}},
\ \dots\ ,\
z_{\la_3+\la_2-4}\sim r_0-\frac{\la_1-1}{2(\la_1-2)}\,\ep^{r_{\la_3+\la_2-4}},
\\[3pt]
& z_{\la_3+\la_2-3}\sim r_0+\sqrt{\frac{\la_1+1}{\la_1-1}}\;
\ep^{(r_{\la_3+\la_2-3}+r_{\la_3+\la_2-2})/2},\
\ z_{\la_3+\la_2-2}\sim r_0-\sqrt{\frac{\la_1+1}{\la_1-1}}\;
\ep^{(r_{\la_3+\la_2-3}+r_{\la_3+\la_2-2})/2},
\\[3pt]
& z_{\la_3+\la_2-1}\sim r_0-\frac{2(\la_1+2)}{\la_1+1}\,\ep^{r_{\la_3+\la_2-1}},
\\[3pt]
& z_{\la_3+\la_2} \sim r_0
-\frac {(\la_1+1)(\la_1-1)}{(\la_1+2)\la_1}\,\ep^{r_{\la_3+\la_2}},\ \dots\ ,\
z_{\la_3+\la_2+\la_1-2} \sim r_0-\frac38\,\ep^{r_{\la_3+\la_2+\la_1-2}},\
\\[3pt]
& z_{\la_3+\la_2+\la_1-1} \sim r_0 +\frac{1}{\sqrt3}\,
\ep^{(r_{\la_3+\la_2+\la_1-1}+r_{\la_3+\la_2+\la_1})/2},\
\ z_{\la_3+\la_2+\la_1} \sim r_0-\frac{1}{\sqrt3}\,
\ep^{(r_{\la_3+\la_2+\la_1-1}+r_{\la_3+\la_2+\la_1})/2}.
\end{align*}
The point \,${\bs T_{\bs r}(\ep,s)\,=\,(z_1,\dots, z_{\la_3+\la_2+\la_1},\,
t^{(1)}_1,\dots,t^{(1)}_{\la_3+\la_2},\,t^{(2)}_1,\dots,t^{(2)}_{\la_3})}$
\,is a point of root coor\-dinates of $X_{\bs r}(\ep,s)$.

Let us call the root coordinates $t^{(2)}_{\la_3-1}$, $t^{(2)}_{\la_3}$,
$t^{(1)}_{\la_3+\la_2}$ exceptional, and the remaining root coordinates
regular.
For each regular root coordinate $y$ the leading
term of asymptotics of $y-r_0$ as $\ep\to 0$ has the form $A\ep^B$ for suitable numbers
$A\ne 0$, $B$.

\begin{lem}
\label{lem z neq t 102n}
The pairs $(A,B)$ are different for different regular root coordinates.
\end{lem}

\begin{proof}
A proof is by inspection of the list.
\end{proof}

For each exceptional coordinate $y$ the the absolute value of the
difference $y-r_0$ is
much smaller as $\ep\to 0$ than for any regular coordinate.

The Bethe vector has the form
$\omega(\bs T_{\bs r}(\ep,s))\,=\,
\sum_J\,\omega_J(\bs T_{\bs r}(\ep,s))\,e_Jv$,\
where the sum is over all admissible $J$, see Section
\ref{Universal weight function and Bethe vectors}. An
admissible $J=(j_1,\dots,j_{\la_3+\la_2+\la_1})$ consists of ones, twos and
threes with exactly $\la_3$ threes and $\la_2$ twos.
Choose $J$ with $j_i=3$ for $i=1,2,\dots, \la_3-3, \la_3-2, \la_3+\la_2-1, \la_3+\la_2$
and $j_i=2$ for $i=\la_3-1, \la_3, \dots , \la_3+\la_2-2$.
Then $\omega_J(\bs T_{\bs r}(\ep,s))$ is given by the formula
\begin{align}
\label{102n}
\omega_J(\bs T_{\bs r}(\ep,s))\ ={}
& \sum_{\sigma \in S_{\la_3+\la_2}}
\sum_{\tau \in S_{\la_3}}\prod_{i=1}^{\la_3-2}
\frac 1{(t^{(2)}_{\tau(i)}-t^{(1)}_{\sigma(i)})
(t^{(1)}_{\sigma(i)}-z_{i})}\times{}
\\[2pt]
&\hphantom{_{\sigma\in{}}}\times
\prod_{i=\la_3-1}^{\la_3}
\frac 1{(t^{(2)}_{\tau(i)}-t^{(1)}_{\sigma(i)})
(t^{(1)}_{\sigma(i)}-z_{\la_2+i})}
\prod_{i=\la_3+1}^{\la_3+\la_2}\frac 1 {t^{(1)}_{\sigma(i)}-z_{i-2}} \ .
\notag
\end{align}
It is easy to see that
$\bar \omega_J(\ep) = \lim_{s\to 0}\,\omega_J(\bs T_{\bs r}(\ep,s))$
is an acceptable function and its order equals
$b=- \sum_{i=1}^{\la_3-2}B(t^{(2)}_i-r_0)
-\sum_{i=1}^{\la_3+\la_2}B(z_i-r_0) - 2 B(t^{(1)}_{\la_3+\la_2-1}-r_0)$.
Namely, consider the following four summands in \Ref{102n}:
\begin{align*}
q\ =\,{} &
\prod_{i=1}^{\la_3-2}
\frac 1 {t^{(2)}_{i}-t^{(1)}_{i}}
\prod_{i=1}^{\la_3+\la_2-4}\frac 1 {t^{(1)}_{i}-z_{i}}
\prod_{\la_3+\la_2-1}^{\la_3+\la_2}\frac 1 {t^{(1)}_{i}-z_{i}} \times{}
\\[3pt]
&{}\times\,\Bigl(\frac 1{(t^{(2)}_{\la_3-1}-t^{(1)}_{\la_3+\la_2-1})
(t^{(2)}_{\la_3}-t^{(1)}_{\la_3+\la_2})} +
\frac 1 {(t^{(2)}_{\la_3-1}-t^{(1)}_{\la_3+\la_2})
(t^{(2)}_{\la_3}-t^{(1)}_{\la_3+\la_2-1})}\Bigr)\times{}
\\[3pt]
&{}\times\,\Bigl(\frac 1 {(t^{(1)}_{\la_3+\la_2-3}-z_{\la_3+\la_2-3})
(t^{(1)}_{\la_3+\la_2-2}-z_{\la_3+\la_2-2})}+{}
\\[2pt]
&\hphantom{{}\!\!\times\Bigl((t^{(1)}_{\la_3+\la_2-3}-z_{\la_3+\la_2-3})}
+\,\frac 1{(t^{(1)}_{\la_3+\la_2-3}-z_{\la_3+\la_2-2})
(t^{(1)}_{\la_3+\la_2-2}-z_{\la_3+\la_2-3})}\Bigr)\ .
\end{align*}
Then the order of $\lim_{s\to 0}q$ equals $b$ and the order of
$\lim_{s\to 0}(\omega_J((\bs T_{\bs r}(\ep,s))-q)$ is greater than $b$.
Therefore, $\lim_{s\to 0}\,\omega_J(\bs T_{\bs r}(\ep,s))$ is nonzero for small $\ep$.

\medskip

For $N=3$ and every essential subset $\W_{\bs d}$,
we proved that the Bethe vector is nonzero
at generic points of $\W_{\bs d}$
and, hence, the number $\al$ of Corollary
\ref{cor al} is nonpositive. Thus, Theorem \ref{thm conj} is proved for $N=3$.

\end{document}